\documentclass[11pt]{amsart}
\usepackage[text={450pt,575pt},headheight=9pt,centering]{geometry}

\usepackage{placeins}

\usepackage{amsthm}
\usepackage{array}
\usepackage{adjustbox}
\usepackage{booktabs}
\usepackage{placeins}

\usepackage{pictexwd,dcpic}
\usepackage{latexsym}
\usepackage[all,cmtip]{xy}
\usepackage[pagebackref]{hyperref}
\usepackage[nameinlink,capitalise]{cleveref}
\usepackage{times}
\usepackage{braket}
\usepackage[english]{babel}
\usepackage[utf8]{inputenc}
\usepackage{faktor}
\usepackage{paralist,url,verbatim}
\usepackage{amscd}
\usepackage[T1]{fontenc}
\usepackage{xspace}
\usepackage{comment}
\usepackage{amsmath,amsthm,amssymb,amsfonts}
\usepackage{mathtools}
\usepackage{enumitem}
\usepackage{wasysym}
\usepackage{booktabs}
\usepackage[usenames,dvipsnames]{xcolor}
\numberwithin{equation}{section}

\definecolor{darkblue}{RGB}{0,0,160}
\definecolor{darkgreen}{RGB}{0,80,0}
\hypersetup{
    colorlinks=true,
    citecolor=darkblue,
    linkcolor=darkblue,
    urlcolor=darkblue,
    filecolor=darkblue,
}

\newcommand{\ktrash}[1]{}

\newcommand{\change}[1]{}
\definecolor{cKlaus}{rgb}{0.15,0.40,0.03}
\definecolor{cKLinkGBU}{rgb}{1,0,0}  
\definecolor{cKLink}{rgb}{0.6,0.2,0.3}
\definecolor{cALink}{rgb}{0,0.3,0}
\usepackage[textsize=tiny]{todonotes}

\usepackage{tikz}
\usetikzlibrary{decorations.markings,intersections,positioning,calc}

  \tikzset{mylabel/.style  args={at #1 #2  with #3}{
    postaction={decorate,
    decoration={
      markings,
      mark= at position #1
      with  \node [#2] {#3};
 } } } }

\theoremstyle{plain}
\newtheorem{theorem}{Theorem}[section]

\newtheorem{corollary}[theorem]{Corollary}
\newtheorem{lemma}[theorem]{Lemma}
\newtheorem{proposition}[theorem]{Proposition}

\theoremstyle{definition}

\newtheorem{definition}[theorem]{Definition}

\newtheorem{example}[theorem]{Example}

\newcommand{\NN}{\ensuremath{\mathbb{N}}}

\newcommand{\RR}{\ensuremath{\mathbb{R}}}
\newcommand{\CC}{\ensuremath{\mathbb{C}}}

\newcommand{\ZZ}{\ensuremath{\mathbb{Z}}}

\newcommand{\bfa}{\ensuremath{\mathbf{a}}}
\newcommand{\bfb}{\ensuremath{\mathbf{b}}}

\newcommand{\bfk}{\ensuremath{\mathbf{k}}}

\newcommand{\bfx}{\ensuremath{\mathbf{x}}}

\newcommand{\kenditem}{\vspace{-1ex}\end{itemize}}

\newcommand{\kitem}{\begin{itemize}\vspace{-2ex}}

\renewcommand{\1}{$1$-rigid}

\newcommand{\KS}{\operatorname{KS}}

\newcommand{\newt}{\Delta}

\usepackage{tikz-cd}

\newcommand{\Def}{\operatorname{Def}}

\title{Unobstructedness of Affine Terminal Gorenstein Toric Fourfolds}

\author{Matej Filip${}^1$}
\address{University of Ljubljana \& Institute of Mathematics, Physics and Mechanics, Trzaska cesta 25, Ljubljana, Slovenia}
\email{matej.filip@fe.uni-lj.si}
\thanks{${}^1$Supported by Slovenian Research Agency program P1-0222 and grants J1-60011, J1-70017.}
\author{Alja\v{z} Zalar${}^2$}
\address{University of Ljubljana \&
Institute of Mathematics, Physics and Mechanics\\ Ve\v cna pot 113\\
1000 Ljubljana\\
Slovenia}
\email{aljaz.zalar@fri.uni-lj.si}
\thanks{${}^2$Supported by Slovenian Research Agency program P1-0288 and grants J1-50002, J1-60011, and J1-70017}

\begin{document}

\begin{abstract}
We prove that every affine terminal Gorenstein toric variety of
dimension at most four is unobstructed. In dimension four, the
obstruction space need not vanish; instead, we determine the possible
homogeneous obstruction degrees and construct simultaneous
two-parameter deformations which eliminate the remaining potential
obstructions.
\end{abstract}

\maketitle

\section{Introduction}

Deformation theory provides a fundamental tool for relating singular
algebraic varieties to smooth ones and for understanding the local
structure of moduli spaces. A variety \(X\) is called
\emph{unobstructed} if its deformation functor \(\Def_X\) is formally
smooth, that is, if every infinitesimal deformation of \(X\) lifts
across every small extension of local Artinian algebras. In
particular, every first-order deformation can then be extended to a
formal deformation of all orders. If \(T_X^1\) is finite-dimensional
and \(\Def_X\) admits a hull, unobstructedness is equivalent to the
smoothness of the miniversal base. Thus unobstructedness is important
both for constructing smoothings and for controlling the local
geometry of deformation and moduli spaces.

Unobstructedness is known in several important geometric settings.
The Bogomolov--Tian--Todorov theorem gives unobstructedness for smooth
Calabi--Yau manifolds, with \(T^1\)-lifting approaches developed by
Ran and Kawamata; see \cite{Ran92b,Kaw92}. For singular Calabi--Yau
varieties, unobstructedness results under various assumptions were
obtained by Kawamata, Ran, Tian, and Namikawa; see
\cite{Kaw92,Ran92a,Tia92,Nam94}. More recent generalisations appear in \cite{FL24b,FL25b}. On the Fano side, every smooth Fano manifold is unobstructed. More generally, Sano proved that every weak Fano manifold is unobstructed \cite{San14}.

For affine toric varieties, unobstructedness can often be established by proving the vanishing of the obstruction space \(T^2_X\). Altmann and Sletsj{\o}e gave a combinatorial description of the relevant Andr\'e--Quillen cohomology groups in \cite{AS98}, while Altmann and van Straten subsequently established polyhedral vanishing criteria for \(T^2_X\), yielding unobstructedness for certain classes of affine Gorenstein toric singularities in \cite{AvS00}. 

The principal result of this paper is that every affine terminal Gorenstein toric variety of dimension at most four is unobstructed; see Theorem~\ref{th term un}. In dimensions at most three, this follows from the vanishing of the obstruction space \(T_X^2\); see Lemma~\ref{lem for terminal dim 3}. In dimension four, however, \(T_X^2\) need not vanish; see Example~\ref{ex eight vertex T2}. Thus, in this case, unobstructedness cannot be deduced solely from the vanishing of the obstruction space.

This unobstructedness result is also relevant to mirror symmetry. It was shown in \cite{Fil26}, for Laurent polynomials in two variables, that if $f$ mutates to $g$ and \(X_{\newt(g)}\) is unobstructed, then the deformation of \(X_{\newt(f)}\) corresponding to the sequence of mutations from \(f\) to \(g\) is unobstructed. The results of the present paper therefore suggest a four-dimensional analogue. Namely, if a Laurent polynomial \(f\) in three variables can be mutated to a Laurent polynomial \(g\) such that \(\newt(g)\) is empty, or equivalently such that \(X_{\newt(g)}\) is terminal (see Definition~\ref{def emp}), then one expects the corresponding mutation deformation of \(X_{\newt(f)}\) to be unobstructed. Establishing this statement requires extending the mutation--deformation comparison of \cite{Fil26}, which was proved only in affine dimension three, to the four-dimensional setting.

We now explain the strategy for proving unobstructedness in dimension four. Let \(X_P\) be a four-dimensional affine terminal Gorenstein toric variety. Then \(P\) is an empty three-dimensional lattice polytope. By Proposition~\ref{th gen},
$
T^2_{X_P}(-r)=0
$
for every \(r\neq R^*\), where \(R^*\) denotes the Gorenstein degree. Thus \(T^2_{X_P}(-R^*)\) is the only homogeneous component of the obstruction space which may be nonzero.

On the other hand, Proposition~\ref{prop empty T1 degrees} determines all nonzero homogeneous components of the tangent space \(T^1_{X_P}\). Besides the component in degree \(-R^*\), they are associated with square facets \(F\subset P\) and occur in degrees
$
-\bigl(R^*-ns_F\bigr),$  $n\geq 1,$
where \(s_F\in\widetilde M\) is the primitive generator of the ray of \(\sigma^\vee\) corresponding to \(F\).

We first consider the finite-dimensional primitive tangent subspace
corresponding to \(n=1\). Its square-facet parameters have degrees
$
m_F:=R^*-s_F, 
$ 
and let \(t_F\) denote the deformation parameter corresponding to a nonzero generator $ \xi_F\in T^1_{X_P}(-m_F).$ Thus $\deg(t_F)=m_F.$
Lemma~\ref{lem square relation} determines exactly when a sum of these
primitive tangent parameter degrees can reach the exceptional
obstruction degree \(R^*\). It shows that the only possibility is
$
m_{F_1}+m_{F_2}=R^*
$
for two parallel square facets \(F_1\) and \(F_2\).

The cohomological vanishing criterion therefore leaves precisely one
type of possible mixed obstruction. To eliminate it, in
Section~\ref{sec compatible two parameter} we construct a simultaneous
two-parameter deformation associated with two compatible deformation
pairs. Applied to the two parallel square facets above, this
construction proves that the coefficient of the possible mixed
quadratic obstruction
$
t_{F_1}t_{F_2}
$
vanishes, even when its target
\(T^2_{X_P}(-R^*)\) is nonzero.

A further difficulty is that \(X_P\) is generally non-isolated and thus 
$
T^1_{X_P}
$
may be infinite-dimensional. Consequently, the full deformation
functor cannot in general be treated by invoking a noetherian
miniversal base. We instead construct compatible formal deformations
over finite-dimensional power-series rings, containing the tangent
directions with \(n\leq N\). The higher tangent directions are
obtained from the primitive ones using the 
module structure on \(T^1_{X_P}\), and the corresponding
families are constructed by a homogeneous substitution in both the
defining equations and their lifted relations. Finally, an induction
over Artinian small extensions shows directly that
\(\Def_{X_P}\) is formally smooth. All these steps are carried out in
the proof of Theorem~\ref{th term un}.

The paper is organised as follows. In the preliminaries, we recall the
graded tangent and obstruction spaces of an affine Gorenstein toric
variety. In Section~\ref{sec compatible two parameter}, we construct
simultaneous two-parameter deformations associated with compatible
deformation pairs. In Section~\ref{sec 4 affine gor ter}, we compute
the tangent and obstruction degrees of affine terminal Gorenstein toric
varieties and prove their unobstructedness in dimensions at most three.
Finally, in Section~\ref{sec 5 unobstructedness}, we prove that every
affine terminal Gorenstein toric fourfold is unobstructed; see
Theorem~\ref{th term un}.

\section{Preliminaries}

\subsection{The setup}\label{sub the setup}

We work over the field \(\mathbb{C}\), which is algebraically closed and of
characteristic zero. Let
$
N\cong \mathbb{Z}^n
$
be a lattice of rank \(n\), and let
$
M:=\operatorname{Hom}_{\mathbb{Z}}(N,\mathbb{Z})
$
be its dual lattice. We write
\[
N_{\mathbb{R}}:=N\otimes_{\mathbb{Z}}\mathbb{R},
\qquad
M_{\mathbb{R}}:=M\otimes_{\mathbb{Z}}\mathbb{R}.
\]

Let \(P\subset N_{\mathbb{R}}\) be a full-dimensional lattice polytope
with vertices \(v^1,\dots,v^p\). Set
\[
\widetilde N:=N\oplus\mathbb{Z},
\qquad
\widetilde M:=M\oplus\mathbb{Z}.
\]
By embedding \(P\) at height \(1\), we obtain the strongly convex rational
polyhedral cone
\begin{equation}\label{polytope}
\sigma
:=
\operatorname{cone}(a^1,\ldots,a^p)
=
\left\{
\sum_{i=1}^p\lambda_i a^i
\;\middle|\;
\lambda_i\in\mathbb{R}_{\geq 0}
\right\}
\subset \widetilde N_{\mathbb{R}},
\end{equation}
where
\[
a^i:=(v^i,1)\in \widetilde N,
\qquad i=1,\dots,p,
\]
and
$
\widetilde N_{\mathbb{R}}
:=
\widetilde N\otimes_{\mathbb{Z}}\mathbb{R}
$. 
The dual cone is
\[
\sigma^\vee
:=
\left\{
r\in\widetilde M_{\mathbb{R}}
\;\middle|\;
\langle n,r\rangle\geq 0
\text{ for every }n\in\sigma
\right\},
\]
where
$
\widetilde M_{\mathbb{R}}
:=
\widetilde M\otimes_{\mathbb{Z}}\mathbb{R}.
$
We consider the affine semigroup
$
S_P:=\sigma^\vee\cap\widetilde M
$
and the associated affine Gorenstein toric variety
$
X_P:=\operatorname{Spec}\mathbb{C}[S_P].
$
Conversely, every affine Gorenstein toric variety without torus factors is
isomorphic to \(X_P\) for some full-dimensional lattice polytope \(P\).

We denote the coordinate projections by
\[
\pi_M:\widetilde M\longrightarrow M,
\qquad
\pi_{\mathbb{Z}}:\widetilde M\longrightarrow\mathbb{Z}.
\]
These projections will be used throughout the paper.

\begin{definition}\label{d:eta(c)}
Let \(Q\subset N_{\mathbb{R}}\) be a polytope and let
\(c\in M_{\mathbb{R}}\). Choose a vertex \(v_Q(c)\) of \(Q\) at which
the function
$
\langle c,\cdot\rangle:Q\longrightarrow\mathbb{R}
$
attains its minimum. We define
\[
\eta_Q(c)
:=
-\min_{v\in Q}\langle v,c\rangle
=
-\langle v_Q(c),c\rangle.
\]
Although the vertex \(v_Q(c)\) need not be unique, the value
\(\eta_Q(c)\) is independent of this choice.
\end{definition}

The Hilbert basis of \(S_P=\sigma^\vee\cap\widetilde M\) can be written as
\begin{equation}\label{eg hilbbas}
H_P
=
\left\{
s_1=(c_1,\eta_P(c_1)),\ldots,
s_r=(c_r,\eta_P(c_r)),R^*
\right\},
\end{equation}
where
$
R^*:=(\underline 0,1),
$
with the convention that \(R^*\) is omitted if it is decomposable in
\(S_P\); see \cite[Section~4.3]{Alt97}. Here
\(c_1,\ldots,c_r\in M\setminus\{\underline 0\}\) are pairwise distinct
and uniquely determined up to ordering. The element \(R^*\) is called
the \emph{Gorenstein degree}.

\begin{example}\label{ex eight vertex polytope hilbert basis}{\footnote{The  file with Magma code for this example numerical computations can be found on the link \url{https://github.com/matfilip/Unobstructedness}.}} 
Let \(N=\ZZ^3\), with coordinates \((x,y,z)\), and let
\[
\begin{aligned}
P=\operatorname{conv}\{&v^1=(0,0,0),\ v^2=(1,0,0),\
 v^3=(0,1,0),\ v^4=(1,1,0),\\
&v^5=(0,0,1),\ v^6=(1,0,1),\
 v^7=(1,1,1),\ v^8=(2,1,1)\}.
\end{aligned}
\]
A direct facet computation gives
\[
\begin{gathered}
x\geq0,\qquad y\geq0,\qquad z\geq0,\qquad
1+x-y-z\geq0,\\
1-x+y\geq0,\qquad 1-x+z\geq0,\qquad
1-y\geq0,\qquad 1-z\geq0.
\end{gathered}
\]
These inequalities force \(0\leq y,z\leq1\). Checking the four possible
pairs \((y,z)\in\{0,1\}^2\) shows that the lattice points of \(P\) are
precisely \(v^1,\ldots,v^8\). In particular, \(P\) is empty, cf.\ Definition \ref{def emp}.

{The corresponding facets, listed in the same order as the inequalities
above, are
\[
\renewcommand{\arraystretch}{1.2}
\begin{array}{c@{\qquad}c}
x=0
&
\operatorname{conv}\{v^1,v^3,v^5\},
\\
y=0
&
F_{y,0}:=\operatorname{conv}\{v^1,v^2,v^6,v^5\},
\\
z=0
&
F_{z,0}:=\operatorname{conv}\{v^1,v^2,v^4,v^3\},
\\
1+x-y-z=0
&
\operatorname{conv}\{v^3,v^5,v^7\},
\\
1-x+y=0
&
\operatorname{conv}\{v^2,v^6,v^8\},
\\
1-x+z=0
&
\operatorname{conv}\{v^2,v^4,v^8\},
\\
y=1
&
F_{y,1}:=\operatorname{conv}\{v^3,v^4,v^8,v^7\},
\\
z=1
&
F_{z,1}:=\operatorname{conv}\{v^5,v^6,v^8,v^7\}.
\end{array}
\]
Thus \(P\) has four triangular facets and four square facets, the latter
being \(F_{y,0},F_{y,1},F_{z,0}\), and \(F_{z,1}\).

Writing \(ij\) for the edge \([v^i,v^j]\), the edge set is
\[
\mathcal E(P)
=
\{12,13,15,24,26,28,34,35,37,48,56,57,68,78\}.
\]
Together with the vertices \(v^1,\ldots,v^8\), these are all the
nonempty proper faces of \(P\).}

The primitive ray generators of \(\sigma^\vee\) are
\[
\begin{array}{llll}
s_1=(1,0,0,0),
& s_2=(0,1,0,0),
& s_3=(0,0,1,0),
& s_4=(1,-1,-1,1),\\[1mm]
s_5=(-1,1,0,1),
& s_6=(-1,0,1,1),
& s_7=(0,-1,0,1),
& s_8=(0,0,-1,1).
\end{array}
\]
{
Here \(s_i\) is the primitive support element corresponding to the
\(i\)-th facet in the preceding list.}
Moreover,
\[
R^*=(0,0,0,1)=s_2+s_7=s_3+s_8.
\]
To determine the Hilbert basis, consider the following triangulation of
\(\sigma^\vee\):
\[
\begin{aligned}
&\operatorname{cone}(R^*,s_6,s_3,s_7),
&&\operatorname{cone}(R^*,s_1,s_3,s_7),
&&\operatorname{cone}(R^*,s_1,s_3,s_2),\\
&\operatorname{cone}(s_8,R^*,s_6,s_7),
&&\operatorname{cone}(s_8,R^*,s_1,s_2),
&&\operatorname{cone}(s_4,R^*,s_1,s_7),\\
&\operatorname{cone}(s_4,s_8,R^*,s_7),
&&\operatorname{cone}(s_4,s_8,R^*,s_1),
&&\operatorname{cone}(s_5,s_8,R^*,s_2),\\
&\operatorname{cone}(s_5,s_8,R^*,s_6),
&&\operatorname{cone}(s_5,R^*,s_3,s_2),
&&\operatorname{cone}(s_5,R^*,s_6,s_3).
\end{aligned}
\]
The determinant of the four generators of each cone is \(\pm1\), so this
triangulation is unimodular. Hence every lattice point of
\(\sigma^\vee\) is generated by \(s_1,\ldots,s_8,R^*\). Since
\(R^*=s_2+s_7=s_3+s_8\), it is decomposable, whereas the \(s_i\) are
primitive generators of extremal rays and hence indecomposable.
Therefore
\[
H_P=\{s_1,\ldots,s_8\}.
\]
\end{example}

\subsection{Graded deformation theory}
\label{sub graded deformation theory}

We briefly recall from \cite[Chapters~2 and~3]{Ser06} the basic notions
concerning deformation functors, small extensions, tangent and
obstruction spaces, and formal smoothness.

Let \(X\) be an affine scheme over \(\CC\), and let
\(\Def_X\) denote its deformation functor. We denote by \(T_X^1\) its
tangent space, so that
\[
T_X^1
\cong
\Def_X\bigl(\CC[\epsilon]/(\epsilon^2)\bigr).
\]

{
Let \((B,\mathfrak m_B)\) be a complete local \(\CC\)-algebra. A
formal deformation of \(X\) over \(B\) is a compatible system of
deformations over the Artinian quotients
\(B/\mathfrak m_B^{k+1}\). Its first-order truncation over
\(B/\mathfrak m_B^2\) determines the Kodaira--Spencer map
\[
\KS_{\mathcal X}\colon
(\mathfrak m_B/\mathfrak m_B^2)^\ast
\longrightarrow T_X^1.
\]
Equivalently, the image of a tangent vector is the class of the
first-order deformation obtained by pulling back along the
corresponding homomorphism
\(B/\mathfrak m_B^2\to\CC[\epsilon]/(\epsilon^2)\).
}

We also denote by \(T_X^2\) a natural obstruction space for
\(\Def_X\). 
More precisely, let
\[
0\longrightarrow J
\longrightarrow B'
\longrightarrow B
\longrightarrow0
\]
be a small extension of local Artinian \(\CC\)-algebras, meaning that
$\mathfrak m_{B'}J=0.$
For every deformation \(\mathcal X_B\) of \(X\) over \(B\), there is
an obstruction class
\[
\operatorname{ob}(\mathcal X_B,B')
\in
T_X^2\otimes_{\CC}J
\]
whose vanishing is equivalent to the existence of a lifting of
\(\mathcal X_B\) to \(B'\). If this obstruction vanishes, the set of
isomorphism classes of liftings is a torsor under
$
T_X^1\otimes_{\CC}J.
$

We say that \(X\) is \emph{unobstructed} if its deformation functor
\(\Def_X\) is formally smooth, that is, if the natural map
$
\Def_X(B')
\longrightarrow
\Def_X(B)
$
is surjective for every small extension \(B'\to B\).

We now specialise to \(X=X_P\). The torus action on \(X_P\) induces
\(\widetilde M\)-gradings
\[
T^i_{X_P}
=
\bigoplus_{r\in\widetilde M}
T^i_{X_P}(-r),
\qquad i=1,2.
\]
Moreover, the obstruction theory is torus-equivariant. Consequently,
obstruction classes decompose into their \(\widetilde M\)-homogeneous
components, and the equations governing every finite-dimensional
graded deformation problem may be chosen
\(\widetilde M\)-homogeneous.

{For a graded deformation, a base parameter of degree \(r\) has
Kodaira--Spencer class in \(T^1_{X_P}(-r)\).}

In the following two subsections, we recall the combinatorial
descriptions of the homogeneous components of \(T^1_{X_P}\) and
\(T^2_{X_P}\). These descriptions will allow us to determine the
possible tangent parameter degrees and the degrees in which
obstructions can occur.

\subsection{The tangent space}

Let \(\sigma\) be as in \eqref{polytope}. For
\(r\in\widetilde M\) and \(j=1,\ldots,p\), recall from
\cite{AS98,Alt00} the subsets
\[
K_j^r
:=
\left\{
m\in S_P
\ \middle|\
\langle a^j,m\rangle<\langle a^j,r\rangle
\right\},
\qquad
S_P=\sigma^\vee\cap\widetilde M.
\]
More generally, for every nonzero face \(\tau\preceq\sigma\), define
\[
K_\tau^r
:=
\bigcap_{a^j\in\tau}K_j^r,
\]
and, for the zero face, put
$
K_0^r:=S_P.
$
Thus, if \(\gamma\preceq\tau\), then
$
K_\tau^r\subseteq K_\gamma^r.
$

The complex \(\operatorname{span}(K^r)_\bullet\) of free abelian
groups is defined by
\[
\operatorname{span}(K^r)_{-k}
:=
\bigoplus_{\substack{\tau\preceq\sigma\\\dim\tau=k}}
\operatorname{span}_{\ZZ}(K_\tau^r),
\]
with differentials induced by the inclusions of faces and the usual
incidence signs. We write
\[
\operatorname{span}(K^r)_\bullet^*
:=
\operatorname{Hom}_{\ZZ}
\bigl(\operatorname{span}(K^r)_\bullet,\ZZ\bigr).
\]

Since \(X_P\) is Gorenstein, and hence Gorenstein in codimension two,
we have
\begin{equation}
\label{eq tk eq}
T^k_{X_P}(-r)
\cong
H^k\left(
\operatorname{span}(K^r)_\bullet^*
\otimes_{\ZZ}\CC
\right),
\qquad
k=1,2;
\end{equation}
see \cite[Propositions~5.3 and~5.4]{AS98} and
\cite[Theorem~2.1]{Alt00}. Here \(T^k_{X_P}(-r)\) denotes the
homogeneous component of degree \(-r\in\widetilde M\). The torus
action induces the decomposition
\[
T^k_{X_P}
=
\bigoplus_{r\in\widetilde M}T^k_{X_P}(-r).
\]

We next recall the more explicit description of \(T^1_{X_P}\)
obtained in \cite{Alt00}. For \(r\in\widetilde M\), write
\[
\varphi_r(v)
:=
\langle r,(v,1)\rangle,
\qquad
v\in N_{\RR},
\]
and consider the cross-section
\[
Q(r)
:=
\sigma\cap
\left\{
x\in\widetilde N_{\RR}
\ \middle|\
\langle r,x\rangle=1
\right\}.
\]

Let \(\mathcal E_c(r)\) denote the set of compact edges of \(Q(r)\).
For every \(E\in\mathcal E_c(r)\), choose an orientation and denote
its direction vector by \(d_E\). For every compact two-dimensional
face \(\varepsilon\subset Q(r)\), choose incidence signs
$
\delta_\varepsilon(E)\in\{0,\pm1\}
$
such that
\[
\sum_{E\subset\varepsilon}
\delta_\varepsilon(E)d_E=0.
\]
Define
\[
V(r)
:=
\left\{
(t_E)_{E\in\mathcal E_c(r)}
\ \middle|\
\sum_{E\subset\varepsilon}
\delta_\varepsilon(E)t_Ed_E=0
\text{ for every compact two-face }\varepsilon
\right\}.
\]

If \(E\in\mathcal E_c(r)\) is induced by an edge
$
[v^i,v^j]\subset P,
$
let \(\ell(E)\) denote the lattice length of this edge. Define
\[
V'(r)
:=
\left\{
(t_E)_E\in V(r)
\ \middle|\
t_E\neq0
\Longrightarrow
1\leq
\varphi_r(v^i)
=
\varphi_r(v^j)
\leq\ell(E)
\right\}.
\]
We write
\[
V(r)_{\CC}:=V(r)\otimes_{\RR}\CC,
\qquad
V'(r)_{\CC}:=V'(r)\otimes_{\RR}\CC.
\]

\begin{theorem}[{\cite[Theorem~4.1]{Alt00}}]
\label{th Altmann T1}
The following statements hold.
\begin{enumerate}
\item Suppose that
$\varphi_r(v)\leq1$
$\text{for every vertex }v\text{ of }P$.
Then
\[
T^1_{X_P}(-r)
\cong
V(r)_{\CC}/\CC\cdot \mathbf 1,
\]
where
$
\mathbf 1=(1,\ldots,1)
$
is the homothety direction.

\item Suppose that
$
\max_{v\in P}\varphi_r(v)\geq2.
$
Then
\[
T^1_{X_P}(-r)
\cong
V'(r)_{\CC}.
\]
In particular, an edge coordinate can be nonzero only on an edge
\(E=[v^i,v^j]\) satisfying
\[
\varphi_r(v^i)
=
\varphi_r(v^j)
=
q_E,
\qquad
1\leq q_E\leq\ell(E).
\]
\end{enumerate}
\end{theorem}

\subsection{The obstruction space}

For every vertex \(v^j\) of \(P\), one has
\begin{equation}
\label{span-possibilities}
\operatorname{span}_{\CC}(K_j^r)
=
\begin{cases}
\widetilde M_{\CC},
&
\langle a^j,r\rangle\geq2,
\\[2mm]
(a^j)^\perp_{\CC},
&
\langle a^j,r\rangle=1,
\\[2mm]
0,
&
\langle a^j,r\rangle\leq0.
\end{cases}
\end{equation}

For a face \(F\preceq P\), write
\begin{equation}\label{eq hat}
\widehat F
:=
\operatorname{cone}
\left\{
(v,1)\mid v\in F
\right\}
\preceq\sigma
\end{equation}
and set
$
K_F^r:=K_{\widehat F}^r.
$
In particular, if \(F=\{v^j\}\), then
$
K_F^r=K_j^r.
$

By \eqref{eq tk eq}, the vector space \(T^2_{X_P}(-r)\) is the middle
cohomology of the complex
\begin{equation}
\label{eq t2 complex}
\bigoplus_{j=1}^p
\bigl(\operatorname{span}_{\CC}K_j^r\bigr)^*
\xrightarrow{\,g_1\,}
\bigoplus_{E\in\mathcal E(P)}
\bigl(\operatorname{span}_{\CC}K_E^r\bigr)^*
\xrightarrow{\,g_2\,}
\bigoplus_{\varepsilon\in\mathcal F_2(P)}
\bigl(\operatorname{span}_{\CC}K_\varepsilon^r\bigr)^*,
\end{equation}
where \(\mathcal E(P)\) and \(\mathcal F_2(P)\) denote the sets of
edges and two-dimensional faces of \(P\), respectively.

Choose an orientation of every edge of \(P\). For an oriented edge
$
E=[v^i,v^j],
$
write
$
d_E:=v^j-v^i.
$
The first map in \eqref{eq t2 complex} is given by
\[
g_1\bigl((\varphi_i)_i\bigr)_E
=
\left.
(\varphi_j-\varphi_i)
\right|_{\operatorname{span}_{\CC}K_E^r}.
\]

For every two-dimensional face \(\varepsilon\subset P\), choose an
orientation of its boundary and incidence signs
$
\delta_E^\varepsilon\in\{0,\pm1\}
$
such that
\[
\sum_{E\subset\varepsilon}
\delta_E^\varepsilon d_E=0.
\]
The second map in \eqref{eq t2 complex} is given by
\[
g_2\bigl((\varphi_E)_E\bigr)_\varepsilon
=
\sum_{E\subset\varepsilon}
\delta_E^\varepsilon
\left.
\varphi_E
\right|_{\operatorname{span}_{\CC}K_\varepsilon^r}.
\]

\subsection{Affine terminal Gorenstein toric singularities}
An \emph{affine lattice automorphism} of \(N_{\mathbb R}\) is a map \[ \varphi\colon N_{\mathbb R}\longrightarrow N_{\mathbb R}, \qquad x\longmapsto A(x)+b, \] where \[ A\in \operatorname{GL}(N) \qquad\text{and}\qquad b\in N. \] Such a map is also called an \emph{affine unimodular transformation}. \begin{definition} Let \(P,P'\subset N_{\mathbb R}\) be lattice polytopes. We say that \(P\) and \(P'\) are \emph{lattice equivalent}, or \emph{unimodularly equivalent}, if there exists an affine lattice automorphism \[ \varphi(x)=A(x)+b, \qquad A\in\operatorname{GL}(N),\quad b\in N, \] such that $\varphi(P)=P'.$ 
\end{definition}

\begin{definition}\label{def emp}
A lattice polytope \(P\subset N_{\RR}\) is called \emph{empty} if
$
P\cap N=\operatorname{Vert}(P),
$
that is, if its only lattice points are its vertices. 
An affine toric Gorenstein variety $X_P$ is \emph{terminal}, if $P$ is empty. 
\end{definition}

Note that terminal $X_P$ has only terminal singularities, see e.g.\ \cite[Proposition 11.4.12]{CLS11}. 
The following vanishing result will be useful.

\begin{theorem}
\label{th AvS comparison}
Let \(P\) be empty. For \(r\in\widetilde M\), write
\[
\varphi_r(v):=\langle r,(v,1)\rangle,
\qquad
F_r:=P\cap(\varphi_r=1).
\]
Then the following hold:
\begin{enumerate}
\item If there exists a vertex \(v\) of \(P\) such that
$
\varphi_r(v)>1,
$
then
\[
T^k_{X_P}(-r)=0,\qquad k=1,2.
\]

\item Suppose that \(\varphi_r\leq1\) on \(P\). If every
\(\ell\)-dimensional face of \(F_r\) is a simplex, then
\[
T^k_{X_P}(-r)=0
\qquad\text{for } k<\ell,\quad k=1,2.
\]

\item If
$
\dim F_r\leq2,
$
then
$
T^2_{X_P}(-r)=0.
$
\end{enumerate}
\end{theorem}

\begin{proof}
This follows from
\cite[Proposition~4.1, Theorem~6.6]{AvS00}.
\end{proof}

\section{Two-parameter deformations associated with compatible
deformation pairs}
\label{sec compatible two parameter}

{We retain the notation of Subsection~\ref{sub the setup} and follow the
sign and grading conventions of \cite[Sections~2 and~3]{Fil26}. For
convenience, we recall below the conventions used in the construction.}
We construct a simultaneous
two-parameter deformation associated with two deformation pairs having
possibly different Minkowski summands, generalizing the one-parameter
construction of \cite[Section~3]{Fil26}. This result will be crucial in proving the unobstructedness of affine terminal Gorenstein toric varieties; see Theorem~\ref{th term un}.

{We use the convention
\(\NN=\{0,1,2,\ldots\}\). For multi-indices
\(\boldsymbol\alpha,\boldsymbol\beta\in\NN^2\), the notation
\(\boldsymbol\alpha\leq\boldsymbol\beta\) means componentwise
inequality. For a lattice polytope \(Q\subset N_{\RR}\), our support
function is
\[
\eta_Q(c):=-\min_{v\in Q}\langle c,v\rangle,
\qquad c\in M_{\RR}.
\]
Thus \(\eta_Q\) is positively homogeneous and subadditive. We say that
a polytope \(A\) is a Minkowski summand of a polytope \(B\) if
\(B=A+C\) for some polytope \(C\). All semigroup algebras and
polynomial presentations below are \(\widetilde M\)-graded, with
\(\deg(\chi^s)=s\).}

We begin with a finite generating set
$
{S_P=\operatorname{cone} (s_1,\ldots,s_r,R^*)},
$
where each
$
s_j=(c_j,\eta_P(c_j))
$
is a boundary element. The set is not required to be minimal and will
be enlarged below, after the two deformation pairs have been fixed. 
{Put \(\bfx=(x_1,\ldots,x_r)\), and let
\[
\Theta_P\colon
\CC[\bfx,u]\longrightarrow\CC[S_P]
\]
be the surjective homomorphism determined by
\[
\Theta_P(x_j)=\chi^{s_j},
\qquad
\Theta_P(u)=\chi^{R^*}.
\]
Setting
$
I_P:=\ker(\Theta_P),
$
we have
\[
X_P
=
\operatorname{Spec}\CC[S_P]
\cong
\operatorname{Spec}\frac{\CC[\bfx,u]}{I_P}.
\]}

For \(\bfk=(k_1,\ldots,k_r)\in\NN^r\), set
\[
s_{\bfk}:=\sum_{j=1}^r k_js_j,
\qquad
\bfx^{\bfk}:=\prod_{j=1}^r x_j^{k_j},
\]
and, for a lattice polytope \(Q\subset N_{\RR}\), define
\[
\eta_Q(\bfk)
:=
\sum_{j=1}^r\eta_Q(k_jc_j)
-
\eta_Q\left(\sum_{j=1}^r k_jc_j\right).
\]

Every \(s\in S_P\) has a unique decomposition
\[
s=\partial_P(s)+n_P(s)R^*,
\qquad
\partial_P(s)-R^*\notin S_P,
\qquad
n_P(s)\in\NN.
\]

For every element \(b\in\partial_P(S_P)\), choose an exponent vector
\[
\partial(b)\in\mathbb N^r
\qquad\text{such that}\qquad
s_{\partial(b)}=b.
\]
After fixing the two deformation pairs, we shall replace these choices
by representatives adapted simultaneously to their two support
functions.
For \(s\in S_P\), define
$
\partial(s):=\partial\bigl(\partial_P(s)\bigr) 
$ 
and put
\[
\chi_P^s:=\bfx^{\partial(s)}u^{n_P(s)}.
\]
{
For \(s=s_{\bfk}\), write
\[
\partial(\bfk):=\partial(s_{\bfk}).
\]
By positive homogeneity of \(\eta_P\), the preceding definition gives
\[
\eta_P(\bfk)
=
\sum_{j=1}^r k_j\eta_P(c_j)
-
\eta_P\left(\sum_{j=1}^r k_jc_j\right)
=
n_P(s_{\bfk}).
\]
Consequently,
\[
s_{\bfk}
=
s_{\partial(\bfk)}+\eta_P(\bfk)R^*,
\qquad
\chi_P^{s_{\bfk}}
=
\bfx^{\partial(\bfk)}u^{\eta_P(\bfk)}.
\]
}

The following elementary presentation lemma will allow us to enlarge the
semigroup generating set without changing the form of the equations or
of their relations.

\begin{lemma}\label{lem presentation arbitrary generators}
For any finite set of boundary generators as above and any choices of
boundary representatives \(\partial(b)\), the binomials
\[
f_{\bfk}
:=
\bfx^{\bfk}-\chi_P^{s_{\bfk}},
\qquad
\bfk\in\NN^r,
\]
generate \(I_P\). 
{Let
$
\mathcal P:=\CC[\bfx,u]
$
and let
\[
\mathcal F
:=
\bigoplus_{\bfk\in\NN^r}\mathcal P e_{\bfk}
\]
be the free \(\mathcal P\)-module with basis
\(\{e_{\bfk}\mid\bfk\in\NN^r\}\). Let
\[
\psi\colon\mathcal F\longrightarrow I_P
\]
be the unique \(\mathcal P\)-linear homomorphism determined by
\[
\psi(e_{\bfk})=f_{\bfk}.
\]
Then \(\psi\) is surjective, and its kernel is generated as a
\(\mathcal P\)-module by the elements
\[
\rho_{\bfa,\bfk}
:=
e_{\bfa+\bfk}
-
\bfx^{\bfa}e_{\bfk}
-
u^{\eta_P(\bfk)}
e_{\partial(\bfk)+\bfa},
\qquad
\bfa,\bfk\in\NN^r.
\]
Equivalently, the first syzygy module of the family
\(\{f_{\bfk}\}_{\bfk\in\NN^r}\) is generated by the
\(\rho_{\bfa,\bfk}\).
}
\end{lemma}

\begin{proof}
For the corresponding statement for the Hilbert-basis presentation, see
\cite[Lemmas~5.3 and~5.6]{ACF22a}. The same argument applies to any
finite set of boundary generators and any fixed choice of boundary
representatives.
\end{proof}

Consequently, the module of relations among the equations \(f_{\bfk}\)
is generated by
\begin{equation}
\label{eq standard relation two summands}
r_{\bfa,\bfk}
:=
f_{\bfa+\bfk}
-
\bfx^{\bfa}f_{\bfk}
-
u^{\eta_P(\bfk)}
f_{\partial(\bfk)+\bfa}
=
0,
\qquad
\bfa,\bfk\in\NN^r.
\end{equation}

{
\begin{example}\label{ex eight vertex polytope coordinate ring}
Continue with the polytope of
Example~\ref{ex eight vertex polytope hilbert basis}. Recall that
\(H_P=\{s_1,\ldots,s_8\}\), whereas
\(R^*=(0,0,0,1)=s_2+s_7=s_3+s_8\). Although \(R^*\) is decomposable
and hence does not belong to the Hilbert basis, we adjoin it as a
redundant semigroup generator and denote the corresponding variable
by \(u\).

Let \(\mathbf e_1,\ldots,\mathbf e_8\) denote the standard basis of
\(\NN^8\), and write
\(f_{ij}:=f_{\mathbf e_i+\mathbf e_j}\). The following additive
relations hold in \(S_P\):
\[
\begin{alignedat}{3}
s_1+s_5&=s_2+R^*,
&\hspace{3em}
s_1+s_6&=s_3+R^*,
&\hspace{3em}
s_1+s_7&=s_3+s_4,
\\
s_1+s_8&=s_2+s_4,
&\hspace{3em}
s_2+s_6&=s_3+s_5,
&\hspace{3em}
s_2+s_7&=R^*,
\\
s_3+s_8&=R^*,
&\hspace{3em}
s_4+s_5&=s_8+R^*,
&\hspace{3em}
s_4+s_6&=s_7+R^*,
\\
s_5+s_7&=s_6+s_8.
\end{alignedat}
\]
Choose the boundary representatives in accordance with the right-hand
sides of these relations. The corresponding binomials are
\[
\begin{alignedat}{3}
f_{15}&=x_1x_5-x_2u,
&\hspace{3em}
f_{16}&=x_1x_6-x_3u,
&\hspace{3em}
f_{17}&=x_1x_7-x_3x_4,
\\
f_{18}&=x_1x_8-x_2x_4,
&\hspace{3em}
f_{26}&=x_2x_6-x_3x_5,
&\hspace{3em}
f_{27}&=x_2x_7-u,
\\
f_{38}&=x_3x_8-u,
&\hspace{3em}
f_{45}&=x_4x_5-x_8u,
&\hspace{3em}
f_{46}&=x_4x_6-x_7u,
\\
f_{57}&=x_5x_7-x_6x_8.
\end{alignedat}
\]
For instance, \(s_1+s_5=s_2+R^*\) gives
\(\partial(\mathbf e_1+\mathbf e_5)=\mathbf e_2\) and
\(\eta_P(\mathbf e_1+\mathbf e_5)=1\), hence
\(f_{15}=x_1x_5-x_2u\). Similarly, \(s_1+s_7=s_3+s_4\) gives
\(\partial(\mathbf e_1+\mathbf e_7)=\mathbf e_3+\mathbf e_4\) and
\(\eta_P(\mathbf e_1+\mathbf e_7)=0\), hence
\(f_{17}=x_1x_7-x_3x_4\).

Let
\[
J:=
(f_{15},f_{16},f_{17},f_{18},f_{26},
 f_{27},f_{38},f_{45},f_{46},f_{57})
\subseteq\CC[x_1,\ldots,x_8,u].
\]
Since every displayed binomial comes from an additive relation in
\(S_P\), we have \(J\subseteq I_P\). We verify the reverse inclusion
using a monomial parametrization.

Consider the unimodular lattice automorphism
\(U\colon\widetilde M\to\widetilde M\) defined by
\(U(a,b,c,d)=(d,a+d,b+d,c+d)\). Its inverse is
\(U^{-1}(\alpha,\beta,\gamma,\delta)
=(\beta-\alpha,\gamma-\alpha,\delta-\alpha,\alpha)\), so
\(U\in\operatorname{GL}_4(\ZZ)\). Moreover, the vectors
\(U(s_1),\ldots,U(s_8),U(R^*)\) have nonnegative coordinates.

Let \(w,p,q,z\) be auxiliary variables corresponding, in this order,
to the four coordinates of \(U(\widetilde M)\). Thus
\((\alpha,\beta,\gamma,\delta)\) is represented by the monomial
\(w^\alpha p^\beta q^\gamma z^\delta\). The transformed generators
determine a monomial homomorphism
\(\Phi\colon\CC[x_1,\ldots,x_8,u]\to\CC[w,p,q,z]\) given by
\[
\begin{array}{c|ccccccccc}
 &x_1&x_2&x_3&x_4&x_5&x_6&x_7&x_8&u\\ \hline
\Phi(\,\cdot\,)
&p&q&z&wp^2&wq^2z&wqz^2&wpz&wpq&wpqz.
\end{array}
\]
For example, \(U(s_1)=(0,1,0,0)\) and
\(U(s_4)=(1,2,0,0)\), which explains the assignments
\(\Phi(x_1)=p\) and \(\Phi(x_4)=wp^2\).

Since \(U\) is unimodular, it preserves all additive relations among
the semigroup generators. Hence
\(\operatorname{im}(\Phi)\cong\CC[S_P]\) and
\(\ker(\Phi)=I_P\). A Gr\"obner-basis computation in Magma gives
\[
\begin{aligned}
\mathcal G=\bigl\{&
x_2x_6x_8-x_5u,\;
x_1x_7x_8-x_4u,\;
x_2x_4-x_1x_8,\;
x_3x_4-x_1x_7,\\
&
x_1x_5-x_2u,\;
x_3x_5-x_2x_6,\;
x_4x_5-x_8u,\;
x_1x_6-x_3u,\\
&
x_4x_6-x_7u,\;
x_2x_7-u,\;
x_5x_7-x_6x_8,\;
x_3x_8-u
\bigr\}
\end{aligned}
\]
for \(I_P\). Ten elements of \(\mathcal G\) are the generators of
\(J\), up to multiplication by \(-1\). The remaining two cubic
binomials are redundant: indeed,
\(x_2x_6x_8-x_5u=x_8f_{26}+x_5f_{38}\) and
\(x_1x_7x_8-x_4u=x_8f_{17}+x_4f_{38}\). Thus
\(\mathcal G\subseteq J\), and consequently \(I_P\subseteq J\).
Together with \(J\subseteq I_P\), this proves \(I_P=J\). Therefore
\[
X_P\cong
\operatorname{Spec}
\frac{\CC[x_1,\ldots,x_8,u]}
{(f_{15},f_{16},f_{17},f_{18},f_{26},
  f_{27},f_{38},f_{45},f_{46},f_{57})}.
\]

We finally illustrate the syzygy statement of
Lemma~\ref{lem presentation arbitrary generators}. Take
\(\bfk=\mathbf e_1+\mathbf e_5\) and \(\bfa=\mathbf e_7\). As observed
above, \(\partial(\bfk)=\mathbf e_2\) and \(\eta_P(\bfk)=1\). The
corresponding standard syzygy is
\[
\rho_{\mathbf e_7,\mathbf e_1+\mathbf e_5}
=
e_{\mathbf e_1+\mathbf e_5+\mathbf e_7}
-x_7e_{\mathbf e_1+\mathbf e_5}
-u e_{\mathbf e_2+\mathbf e_7}
\in\mathcal F.
\]
Applying the \(\mathcal P\)-linear homomorphism
\(\psi\colon\mathcal F\to I_P\) gives
\[
\psi\left(
\rho_{\mathbf e_7,\mathbf e_1+\mathbf e_5}
\right)
=
f_{157}-x_7f_{15}-uf_{27},
\]
where
\(f_{157}:=f_{\mathbf e_1+\mathbf e_5+\mathbf e_7}
=x_1x_5x_7-u^2\), since \(s_1+s_5+s_7=2R^*\). Indeed,
\[
\begin{aligned}
f_{157}-x_7f_{15}-uf_{27}
&=(x_1x_5x_7-u^2)
-x_7(x_1x_5-x_2u)
-u(x_2x_7-u)
=0.
\end{aligned}
\]
Thus
\(\rho_{\mathbf e_7,\mathbf e_1+\mathbf e_5}\in\ker(\psi)\), as
asserted by Lemma~\ref{lem presentation arbitrary generators}.
\end{example}
}

{We now recall the notation used in the definition of a deformation
pair. For \(m\in\widetilde M\), let
\[
\varphi_m\colon N_{\RR}\longrightarrow\RR,
\qquad
\varphi_m(v)
:=
\langle m,(v,1)\rangle
=
\langle\pi_M(m),v\rangle+\pi_{\ZZ}(m).
\]}

{
\begin{definition}[\cite{Fil26}]
\label{def deformation pair}
Let \(m\in\widetilde M\), and let
\(Q\subset\ker\bigl(\pi_M(m)\bigr)\) be a lattice polytope. The pair
\((m,Q)\) is called a \emph{deformation pair of \(P\)} if, for every
\(i\in\NN\) such that the section
\[
P_i(m):=P\cap\{\varphi_m=i\}
\]
is nonempty, there exists a polytope \(R_i\) such that
\[
P_i(m)=iQ+R_i.
\]
Here \(iQ=Q+\ldots+Q\) denotes the \(i\)-fold Minkowski sum of \(Q\) for \(i\geq1\),
and \(0Q:=\{0\}\). Equivalently, the condition says that \(iQ\) is a
Minkowski summand of every nonempty integral-level section \(P_i(m)\).
\end{definition}
}

{A deformation pair \((m,Q)\) determines a one-parameter deformation of
\(X_P\) \cite[Corollary 3.9]{Fil26}.} We denote its deformation parameter
by \(t_{(m,Q)}\) and assign it degree
$
\deg t_{(m,Q)}=m.
$
{The deformation-pair condition implies the following regularity property:
\begin{equation}
\label{eq one parameter regularity}
s_{\mathbf k}-jm\in\sigma^\vee
\qquad
\text{whenever }0\le j\le\eta_Q(\mathbf k).
\end{equation}
Indeed, this follows by evaluating on the generators $(v,1)$ of
$\sigma$ and using the Minkowski decomposition
$P_i(m)=iQ+R_i$.}

Let
\[
\mathcal D
=
\bigl\{(m_1,Q_1),(m_2,Q_2)\bigr\}
\]
be two deformation pairs of \(P\). We assume that they are
\emph{mutually compatible}, meaning that
\begin{equation}
\label{eq mutually orthogonal summands}
Q_1,Q_2
\subset
\ker\pi_M(m_1)\cap\ker\pi_M(m_2).
\end{equation}

We also assume that
\begin{equation}
\label{eq disjoint positive loci}
\varphi_{m_1}(v)\varphi_{m_2}(v)=0
\qquad
\text{for every vertex }v\text{ of }P.
\end{equation}

We now make the boundary choices above simultaneously compatible with
\(Q_1\) and \(Q_2\). The following lemma is the reason for allowing a
nonminimal semigroup generating set.

\begin{lemma}\label{lem adapted boundary representatives}
After adjoining finitely many boundary elements to
\(s_1,\ldots,s_r\), every boundary element
\(b\in\partial_P(S_P)\) admits a representative \(\partial(b)\) such
that
\begin{equation}\label{eq adapted boundary vanishing}
\eta_{Q_i}(\partial(b))=0,
\qquad i=1,2.
\end{equation}
The enlarged elements together with \(R^*\) still generate \(S_P\),
and the presentation and syzygy statement of
Lemma~\ref{lem presentation arbitrary generators} applies to this
enlarged presentation.
\end{lemma}

\begin{proof}
Let \(\Sigma\) be a finite common rational polyhedral refinement of the
normal fans of \(P,Q_1\), and \(Q_2\). Equivalently, the three support
functions
$
\eta_P,~\eta_{Q_1},~ \eta_{Q_2}
$
are linear on every cone \(\tau\in\Sigma\). For such a cone, set
\[
B_\tau
:=
\bigl\{(c,\eta_P(c))\mid c\in\tau\cap M\bigr\}
\subset\partial_P(S_P).
\]
Since \(\eta_P\) is linear on \(\tau\), the projection to \(M\)
identifies \(B_\tau\) with the affine semigroup \(\tau\cap M\).
Therefore \(B_\tau\) has a finite Hilbert basis. Adjoin the Hilbert
bases of \(B_\tau\) for all cones \(\tau\in\Sigma\). This adds only
finitely many boundary generators.

Let
\(b=(c,\eta_P(c))\in\partial_P(S_P)\). Choose a cone
\(\tau\in\Sigma\) containing \(c\), express \(b\) as a sum of elements
of the Hilbert basis of \(B_\tau\), and use this expression to define
\(\partial(b)\). All generators occurring in this representative have
first coordinates in the same cone \(\tau\). Since each
\(\eta_{Q_i}\) is linear on \(\tau\),
\[
\eta_{Q_i}(\partial(b))
=
\sum_j k_j\eta_{Q_i}(c_j)
-
\eta_{Q_i}\left(\sum_jk_jc_j\right)
=
0,
\qquad i=1,2.
\]
Finally, every \(s\in S_P\) has the form
$
s=(c,\eta_P(c))+nR^*,
$
$ 
n\in\NN,
$
so the enlarged boundary elements together with \(R^*\) generate
\(S_P\). Lemma~\ref{lem presentation arbitrary generators} concludes the proof.
\end{proof}

From now on, we use this enlarged generating set and these adapted
representatives, retaining the notation \(s_j,x_j,\bfx\), and
\(\partial\). The enlargement only introduces redundant toric
coordinates and does not change \(X_P\); such coordinates may be
eliminated in explicit smaller presentations.

The adapted vanishing also gives, for every
\(\bfa,\bfk\in\NN^r\),
\begin{equation}\label{eq adapted eta additivity}
\eta_{Q_i}(\bfa+\bfk)
=
\eta_{Q_i}(\bfk)
+
\eta_{Q_i}(\partial(\bfk)+\bfa),
\qquad i=1,2.
\end{equation}
Indeed, the projections to \(M\) of \(s_{\bfk}\) and
\(s_{\partial(\bfk)}\) agree, and
\eqref{eq adapted boundary vanishing} applies to
\(\partial(\bfk)\).

Write
\[
t_i:=t_{(m_i,Q_i)},
\qquad
\mathbf t:=(t_1,t_2),
\qquad
\deg(t_i)=m_i.
\]
For
$
\boldsymbol\alpha=(\alpha_1,\alpha_2)\in\NN^2,
$
set
\[
\mathbf t^{\boldsymbol\alpha}
:=
t_1^{\alpha_1}t_2^{\alpha_2},
\qquad
m_{\boldsymbol\alpha}
:=
\alpha_1m_1+\alpha_2m_2.
\]
For \(\bfk\in\NN^r\), write
\[
A_i(\bfk):=\eta_{Q_i}(\bfk),
\qquad
\mathbf A(\bfk):=
\bigl(A_1(\bfk),A_2(\bfk)\bigr),
\]
and
\[
\binom{\mathbf A(\bfk)}{\boldsymbol\alpha}
:=
\binom{A_1(\bfk)}{\alpha_1}
\binom{A_2(\bfk)}{\alpha_2}.
\]
We use the convention that this coefficient is zero if
\(\alpha_i>A_i(\bfk)\) for some \(i\).

For every \(\bfk\in\NN^r\), define
\begin{equation}
\label{eq two summand Fk}
F_{\bfk}(\bfx,u,t_1,t_2)
:=
\bfx^{\bfk}
-
\sum_{\substack{\boldsymbol\alpha\in\NN^2\\
\alpha_i\leq A_i(\bfk),\,i=1,2}}
\binom{\mathbf A(\bfk)}{\boldsymbol\alpha}
\mathbf t^{\boldsymbol\alpha}
\chi_P^{s_{\bfk}-m_{\boldsymbol\alpha}}.
\end{equation}

The characters occurring in \eqref{eq two summand Fk} are well-defined.
Indeed, let \(v\) be a vertex of \(P\). By
\eqref{eq disjoint positive loci}, one of
\(\varphi_{m_1}(v)\) and \(\varphi_{m_2}(v)\) is zero. If, for example,
\(\varphi_{m_2}(v)=0\), then
$
s_{\bfk}-\alpha_1m_1\in\sigma^\vee
$
by \eqref{eq one parameter regularity}, and hence
\[
\begin{aligned}
\left\langle
s_{\bfk}-m_{\boldsymbol\alpha},(v,1)
\right\rangle
&=
\left\langle
s_{\bfk}-\alpha_1m_1,(v,1)
\right\rangle
-
\alpha_2\varphi_{m_2}(v)
\geq0.
\end{aligned}
\]
The other case is identical. Thus
$
s_{\bfk}-m_{\boldsymbol\alpha}\in\sigma^\vee.
$

For every admissible \(\boldsymbol\alpha\), set
\[
\mathbf k_{\boldsymbol\alpha}
:=
\partial\bigl(s_{\mathbf k}-m_{\boldsymbol\alpha}\bigr)
\in\NN^r
\]
and
$
\nu_{\mathbf k,\boldsymbol\alpha}
:=
n_P\bigl(s_{\mathbf k}-m_{\boldsymbol\alpha}\bigr).
$
Thus
\begin{equation} \label{eq two summand boundary decomposition}
s_{\mathbf k}-m_{\boldsymbol\alpha}
=
s_{\mathbf k_{\boldsymbol\alpha}}
+
\nu_{\mathbf k,\boldsymbol\alpha}R^*.
\end{equation}

For \(\bfa,\bfk\in\NN^r\), define
\begin{multline}
\label{eq two summand R}
R_{\bfa,\bfk}
:=
F_{\bfa+\bfk}
-
\bfx^{\bfa}F_{\bfk}
-
u^{\eta_P(\bfk)}F_{\partial(\bfk)+\bfa}
-
\sum_{\substack{\boldsymbol\alpha\in\NN^2\setminus\{\mathbf0\}\\
\alpha_i\leq A_i(\bfk),\,i=1,2}}
u^{\nu_{\bfk,\boldsymbol\alpha}}
\binom{\mathbf A(\bfk)}{\boldsymbol\alpha}
\mathbf t^{\boldsymbol\alpha}
F_{\bfk_{\boldsymbol\alpha}+\bfa}.
\end{multline}

\begin{lemma}
\label{lem eta compatible boundary}
For every \(\bfa\in\NN^r\),
\[
\eta_{Q_i}(\bfk_{\boldsymbol\alpha}+\bfa)
=
\eta_{Q_i}(\partial(\bfk)+\bfa),
\qquad i=1,2.
\]
\end{lemma}

\begin{proof}
Fix \(i\in\{1,2\}\), and write
\[
c_{\bfb}:=\sum_{j=1}^r b_jc_j
\qquad
\text{for }\bfb\in\NN^r.
\]
Since \(\partial(\bfk)\) and \(\bfk_{\boldsymbol\alpha}\) are the chosen
adapted representatives of boundary elements,
Lemma~\ref{lem adapted boundary representatives} gives
\[
\eta_{Q_i}(\partial(\bfk))
=
\eta_{Q_i}(\bfk_{\boldsymbol\alpha})
=
0.
\]
Moreover, projecting the boundary decomposition to \(M\) yields
$
c_{\bfk_{\boldsymbol\alpha}}
=
c_{\partial(\bfk)}
-
\pi_M(m_{\boldsymbol\alpha}).
$
By mutual compatibility,
\[
Q_i\subset\ker\pi_M(m_1)\cap\ker\pi_M(m_2),
\]
and hence \(\pi_M(m_{\boldsymbol\alpha})\) vanishes identically on
\(Q_i\). Therefore
$
\eta_{Q_i}(c_{\bfk_{\boldsymbol\alpha}})
=
\eta_{Q_i}(c_{\partial(\bfk)})
$
and
\[
\eta_{Q_i}(c_{\bfa}+c_{\bfk_{\boldsymbol\alpha}})
=
\eta_{Q_i}(c_{\bfa}+c_{\partial(\bfk)}).
\]
Using the definition of \(\eta_{Q_i}(\,\cdot\,)\) and the two vanishing
equalities above, we conclude that
\[
\eta_{Q_i}(\bfk_{\boldsymbol\alpha}+\bfa)
=
\eta_{Q_i}(\partial(\bfk)+\bfa),
\]
as required.
\end{proof}

\begin{proposition}
\label{pro two summand linear relations}
For every \(\bfa,\bfk\in\NN^r\), the expression
\(R_{\bfa,\bfk}\) is a linear relation among the equations
\(F_{\bfb}\).
\end{proposition}

\begin{proof}
Set
\[
A_i:=\eta_{Q_i}(\bfk),
\qquad
B_i:=\eta_{Q_i}(\partial(\bfk)+\bfa),
\qquad i=1,2.
\]
By \eqref{eq adapted eta additivity},
$
\eta_{Q_i}(\bfa+\bfk)=A_i+B_i.
$
Moreover, Lemma \ref{lem eta compatible boundary} implies
\[
\eta_{Q_i}(\bfk_{\boldsymbol\alpha}+\bfa)=B_i,
\qquad i=1,2.
\]

Expanding \eqref{eq two summand R} and using
\eqref{eq two summand boundary decomposition}, all terms of degree
\(\mathbf t^{\boldsymbol\beta}\) are multiples of the same character.
Their common coefficient is
\[
-
\binom{A_1+B_1}{\beta_1}
\binom{A_2+B_2}{\beta_2}
+
\sum_{\boldsymbol\alpha\leq\boldsymbol\beta}
\binom{A_1}{\alpha_1}
\binom{A_2}{\alpha_2}
\binom{B_1}{\beta_1-\alpha_1}
\binom{B_2}{\beta_2-\alpha_2},
\]
{which is zero by the Vandermonde's convolution applied twice \cite[Equation~(5.27), pp.~170]{GKP94}.}
Hence \(R_{\bfa,\bfk}=0\).
\end{proof}

\begin{corollary}
\label{cor two compatible deformations}
The equations \(F_{\bfk}\) define a formal two-parameter deformation over 
$
\CC[[t_1,t_2]]
$
of \(X_P\). 
The restriction to the \(t_i\)-axis is the one-parameter deformation
associated with \((m_i,Q_i)\), for \(i=1,2\).
\end{corollary}

\begin{proof}
Modulo \((t_1,t_2)\), the equations \(F_{\bfk}\) reduce to
\(f_{\bfk}\), and the relations \(R_{\bfa,\bfk}\) reduce to
\(r_{\bfa,\bfk}\). By
Lemma~\ref{lem presentation arbitrary generators}, the latter generate
all relations among the equations of the central fibre, while
Proposition~\ref{pro two summand linear relations} provides their
lifts. {By the equational flatness criterion
\cite[Lemma~3.5]{Fil26}, the equations therefore define a
formal flat deformation over $\CC[[t_1,t_2]]$.}
\end{proof}
\medskip

{For every facet \(F\subset P\), let \(s_F\in\widetilde M\) denote the
primitive generator of the ray of \(\sigma^\vee\) corresponding to \(F\).
Equivalently, \(s_F\) is the primitive inward support element satisfying
\[
\langle s_F,(v,1)\rangle=0
\qquad\text{for every }v\in F.
\]}

\begin{corollary}
\label{cor parallel squares simultaneous deformation}
Let \(P\) be an empty three-dimensional lattice polytope, and let
\(F_1,F_2\subset P\) be square facets satisfying
\(s_{F_1}+s_{F_2}=R^*\). For \(i=1,2\), let
\((R^*-s_{F_i},Q_i)\) be a deformation pair of \(P\). Then the
corresponding one-parameter deformations extend to a formal
two-parameter deformation over \(\CC[[t_1,t_2]]\).
\end{corollary}

\begin{proof}
The relation \(s_{F_1}+s_{F_2}=R^*\) gives
\(R^*-s_{F_1}=s_{F_2}\), \(R^*-s_{F_2}=s_{F_1}\), and
$$(R^*-s_{F_1})+(R^*-s_{F_2})=R^*.$$ Since
\(s_{F_1},s_{F_2}\in\sigma^\vee\), for every vertex \(v\) of \(P\),
the values \(\varphi_{R^*-s_{F_1}}(v)\) and
\(\varphi_{R^*-s_{F_2}}(v)\) are nonnegative integers whose sum is
\(\varphi_{R^*}(v)=1\). Hence
\(\varphi_{R^*-s_{F_1}}(v)\varphi_{R^*-s_{F_2}}(v)=0\).

Moreover, since \(\pi_M(R^*)=0\), we have
\(\pi_M(R^*-s_{F_1})=-\pi_M(R^*-s_{F_2})\), and therefore
$$\ker\pi_M(R^*-s_{F_1})=\ker\pi_M(R^*-s_{F_2}).$$
Together with \(Q_i\subset\ker\pi_M(R^*-s_{F_i})\) for \(i=1,2\),
this shows that the two deformation pairs satisfy the compatibility
conditions. The result now follows from
Corollary~\ref{cor two compatible deformations}.
\end{proof}

{
The following example makes the preceding construction explicit for the
eight-vertex polytope of
Example~\ref{ex eight vertex polytope hilbert basis}. Its two pairs of
opposite square facets satisfy the hypotheses of
Corollary~\ref{cor parallel squares simultaneous deformation}, and we
write down the resulting two-parameter deformations explicitly.
}

\begin{example}\label{ex eight vertex two parameter deformations}
Let \(P\) be the polytope of
Example~\ref{ex eight vertex polytope hilbert basis}, retain the notation
\(F_{y,0},F_{y,1},F_{z,0},F_{z,1}\) for its four square facets, and use
the coordinate-ring presentation of
Example~\ref{ex eight vertex polytope coordinate ring}. Recall that
their primitive support elements are
\(s_{F_{y,0}}=s_2\), \(s_{F_{y,1}}=s_7\),
\(s_{F_{z,0}}=s_3\), and \(s_{F_{z,1}}=s_8\).

In particular,
\[
s_{F_{y,0}}+s_{F_{y,1}}=s_2+s_7=R^*,
\qquad
s_{F_{z,0}}+s_{F_{z,1}}=s_3+s_8=R^*.
\]

Let \(e_1,e_2,e_3\) be the standard basis of \(N=\ZZ^3\), and set
\(L:=[0,e_1]\). The decompositions
\[
\begin{alignedat}{2}
F_{z,0}&=[0,e_1]+[0,e_2],
&\qquad\qquad
F_{z,1}&=e_3+[0,e_1]+[0,e_1+e_2],
\\
F_{y,0}&=[0,e_1]+[0,e_3],
&\qquad\qquad
F_{y,1}&=e_2+[0,e_1]+[0,e_1+e_3].
\end{alignedat}
\]
show that \(L\) is a Minkowski summand of each square facet. For
\(m_F:=R^*-s_F\), the only nontrivial Minkowski-summand condition in
Definition~\ref{def deformation pair} is the one for the level-one
section
\[
P\cap(\varphi_{m_F}=1)=F;
\]
the condition at level zero is automatic because \(0L=\{0\}\).
Consequently, \((m_F,L)\) is a deformation pair for each of the four
square facets. The two relations above and
Corollary~\ref{cor parallel squares simultaneous deformation} therefore
give one two-parameter deformation for each pair of opposite square
facets.

For \(c=(c_x,c_y,c_z)\in M\), one has
\(\eta_L(c)=\max\{0,-c_x\}\). Consequently, for
\(\bfk=(k_1,\ldots,k_8)\),
\begin{equation}\label{eq eta L eight vertex}
\eta_L(\bfk)
=
\min\{k_1+k_4,k_5+k_6\}.
\end{equation}
Among the ten generators \(f_{ij}\) of the toric ideal displayed in
Example~\ref{ex eight vertex polytope coordinate ring}, the equality
\(\eta_L(\mathbf e_i+\mathbf e_j)=1\) holds precisely for
\(f_{15}, f_{16}, f_{45}\), and \(f_{46}\).

It is zero for the remaining six generators, which therefore remain
unchanged in both families.

\smallskip
\noindent\emph{The pair \(F_{z,0},F_{z,1}\).}
The corresponding deformation degrees are
\[
m_{z,0}:=R^*-s_3=s_8,
\qquad
m_{z,1}:=R^*-s_8=s_3.
\]
Let \(\tau_0,\tau_1\) be the associated parameters, with
\(\deg(\tau_0)=s_8\) and \(\deg(\tau_1)=s_3\).
Formula~\eqref{eq two summand Fk} gives
\[
\begin{aligned}
F^{z}_{15}
&=x_1x_5-x_2u-\tau_0x_2x_3-\tau_1x_2x_8-\tau_0\tau_1x_2,
\\
F^{z}_{16}
&=x_1x_6-x_3u-\tau_0x_3^2-\tau_1u-\tau_0\tau_1x_3,
\\
F^{z}_{45}
&=x_4x_5-x_8u-\tau_0u-\tau_1x_8^2-\tau_0\tau_1x_8,
\\
F^{z}_{46}
&=x_4x_6-x_7u-\tau_0x_3x_7-\tau_1x_7x_8-\tau_0\tau_1x_7.
\end{aligned}
\]

\smallskip
\noindent\emph{The pair \(F_{y,0},F_{y,1}\).}
The corresponding deformation degrees are
\[
m_{y,0}:=R^*-s_2=s_7,
\qquad
m_{y,1}:=R^*-s_7=s_2.
\]
Let \(\lambda_0,\lambda_1\) be the associated parameters, with
\(\deg(\lambda_0)=s_7\) and \(\deg(\lambda_1)=s_2\).
In this case, formula~\eqref{eq two summand Fk} gives
\[
\begin{aligned}
F^{y}_{15}
&=x_1x_5-x_2u
-\lambda_0x_2^2-\lambda_1u-\lambda_0\lambda_1x_2,
\\
F^{y}_{16}
&=x_1x_6-x_3u
-\lambda_0x_2x_3-\lambda_1x_3x_7-\lambda_0\lambda_1x_3,
\\
F^{y}_{45}
&=x_4x_5-x_8u
-\lambda_0x_2x_8-\lambda_1x_7x_8-\lambda_0\lambda_1x_8,
\\
F^{y}_{46}
&=x_4x_6-x_7u
-\lambda_0u-\lambda_1x_7^2-\lambda_0\lambda_1x_7.
\end{aligned}
\]
Setting the two parameters equal to zero recovers the defining
equations of \(X_P\) in each case.
\end{example}

\section{Affine terminal Gorenstein toric varieties}\label{sec 4 affine gor ter}

In this section, we study the tangent and obstruction spaces of affine terminal Gorenstein toric varieties. We give a combinatorial description of the nonzero homogeneous components of $T_X^1$ and $T_X^2$, and show that $T_X^2=0$ when \(\dim X\leq3\). In particular, affine terminal Gorenstein toric varieties of dimension at most three are unobstructed.

The following two lemmas are well known; we include brief proofs for
convenience.

\begin{lemma}\label{lem empty polytope vertices}
Let \(P\subset N_{\mathbb R}\) be an empty lattice polytope of
dimension \(d\), that is,
$
P\cap N=\operatorname{Vert}(P).
$
Then
$
\#\operatorname{Vert}(P)\leq 2^d.
$
\end{lemma}

\begin{proof}
After replacing the ambient lattice by
$
\operatorname{span}_{\mathbb R}(P-P)\cap N,
$
we may assume that \(N\cong \mathbb Z^d\) and that \(P\) is
full-dimensional. There are exactly \(2^d\) congruence classes in
$
N/2N\cong(\mathbb Z/2\mathbb Z)^d.
$
If \(P\) had more than \(2^d\) vertices, then two distinct vertices
\(v,w\) would be congruent modulo \(2N\). Hence
$
\frac{v+w}{2}\in N.
$
Since this point lies in the relative interior of the segment
\([v,w]\subset P\), it is a non-vertex lattice point of \(P\), which
contradicts the assumption that \(P\) is empty.
\end{proof}

\begin{lemma}\label{lem empty lattice polygons}
Let \(P\subset N_{\mathbb R}\), where \(N\simeq\mathbb Z^2\), be an
empty two-dimensional lattice polytope. Then \(P\) is lattice
equivalent to either the standard triangle
\[
\Delta_2:=\operatorname{conv}\{(0,0),(1,0),(0,1)\}
\]
or the unit square
\[
\square:=\operatorname{conv}\{(0,0),(1,0),(0,1),(1,1)\}.
\]
\end{lemma}

\begin{proof}
By Lemma~\ref{lem empty polytope vertices}, \(P\) has either three or
four vertices. Since \(P\) is empty, every edge is primitive and
\(P\) has no interior lattice points. {Thus Pick's formula \cite[Theorem~2.8]{BR15}} gives
$
\operatorname{Area}(P)=\frac{r}{2}-1,
$
where \(r\) is the number of vertices of \(P\).

{
If \(r=3\), then \(\operatorname{Area}(P)=\frac12\). Hence the normalized
area of \(P\) is one, so \(P\) is a unimodular triangle and therefore
lattice equivalent to \(\Delta_2\).

Suppose that \(r=4\), and write the vertices as
\(v_0,v_1,v_2,v_3\) in cyclic order. Then
\(\operatorname{Area}(P)=1\). Each diagonal divides \(P\) into two
lattice triangles. Since every nondegenerate lattice triangle has area
at least \(\frac12\), all four triangles arising from the two diagonals
have area \(\frac12\) and are therefore unimodular.

After a lattice equivalence, we may assume that
\(v_0=0\), \(v_1=e_1\), and \(v_2=e_2\). Writing \(v_3=(a,b)\), the
unimodularity and cyclic ordering of
\(\operatorname{conv}\{v_0,v_2,v_3\}\) and
\(\operatorname{conv}\{v_0,v_1,v_3\}\) give \(-a=1\) and \(b=1\).
Thus \(v_3=e_2-e_1\), and \(P\) is lattice equivalent to the unit
square.}
\end{proof}

\begin{lemma}\label{lem for terminal dim 3}
Let \(X\) be an affine terminal Gorenstein toric variety. If
$
\dim X\le 3,
$
then $T^2_X=0$ and thus \(X\) is unobstructed.
\end{lemma}

\begin{proof}
If \(\dim X\leq2\), then \(X\cong \mathbb A^1\) or \(X\cong \mathbb A^2\). Hence $X$ is smooth
and \(T_X^2=0\). 
{Suppose now that \(\dim X=3\). Then the associated Gorenstein polytope
\(P\) is two-dimensional and, by terminality, empty.}
By Definition~\ref{def emp} and Lemma~\ref{lem empty lattice polygons}, \(P\) is lattice equivalent either to the standard triangle \(\Delta_2\)
or to the unit square \(\square\).
{If \(P\) is lattice equivalent to \(\Delta_2\), then the primitive
generators of the cone over \(P\) form a lattice basis. Hence the cone
is unimodular and
$
X\cong\mathbb A^3; 
$
thus $T_X^2=0$.
If \(P\) is lattice equivalent to the unit square, then the Hilbert
basis of \(\sigma^\vee\cap\widetilde M\) consists of four elements
\(u_1,u_2,u_3,u_4\), with the unique primitive relation
\(u_1+u_4=u_2+u_3\). Consequently,
\[
X\cong
\operatorname{Spec}
\frac{\CC[x_1,x_2,x_3,x_4]}
     {(x_1x_4-x_2x_3)},
\]
which is the three-dimensional ordinary double point. In particular,
\(X\) is a hypersurface and \(T_X^2=0\). So
\(X\) is unobstructed.}
\end{proof}

If \(X\) is a four-dimensional affine terminal Gorenstein toric
variety, then the obstruction space \(T^2_X\) need not vanish, as the
following example shows. Thus, unlike in dimensions at most three,
unobstructedness cannot be deduced simply from the vanishing of
\(T^2_X\); one has to show that the actual obstruction equations
vanish.

\begin{example}\label{ex eight vertex T2}
Let \(P\) be the empty three-dimensional lattice polytope from
Example~\ref{ex eight vertex polytope hilbert basis}, and let \(X_P\)
be its associated four-dimensional affine toric variety. Retain the
notation \(\mathcal E(P)\) for its fourteen edges and
\(F_{y,0},F_{y,1},F_{z,0},F_{z,1}\) for its four square facets. The
remaining four facets are triangles. We compute
\(T^2_{X_P}(-R^*)\).

Recall that \(a^i=(v^i,1)\). Since
\(\langle R^*,a^i\rangle=1\) for every vertex \(v^i\), the complex
computing \(T^2_{X_P}(-R^*)\) is
\begin{equation}\label{eq eight vertex T2 complex}
\bigoplus_{i=1}^{8}\bigl((a^i)^\perp\bigr)^*
\xrightarrow{g_1}
\bigoplus_{E=[v^i,v^j]\in\mathcal E(P)}
\bigl((a^i,a^j)^\perp\bigr)^*
\xrightarrow{g_2}
\bigoplus_{F\in\mathcal F_2(P)}
\bigl(\widehat F^\perp\bigr)^*.
\end{equation}
Each vertex summand has dimension three, each edge summand has dimension
two, and each facet summand has dimension one. Thus the dimensions of
the three terms are \(24\), \(28\), and \(8\), respectively.

We first determine \(V(R^*)\). 
Recall that the closure relation for a two-dimensional face \(F\) is equal to
$
\sum_{E\subset F}\delta_E^F t_Ed_E=0.
$
It expresses the condition that the scaled oriented edge vectors
\(t_Ed_E\) form a closed polygon. For a triangular facet, it forces all
three edge parameters to be equal, whereas for a square facet it
identifies the parameters on opposite edges.
For a triangular facet, the closure
relation forces its three edge parameters to be equal. The four
triangular facets therefore give
\(t_{13}=t_{15}=t_{35}=t_{37}=t_{57}\) and
\(t_{24}=t_{26}=t_{28}=t_{48}=t_{68}\).

For each square facet, the closure relation identifies the parameters
on opposite edges:
\[
\begin{aligned}
F_{y,0}:&\quad t_{12}=t_{56},\qquad t_{15}=t_{26},
\\
F_{z,0}:&\quad t_{12}=t_{34},\qquad t_{13}=t_{24},
\\
F_{y,1}:&\quad t_{34}=t_{78},\qquad t_{37}=t_{48},
\\
F_{z,1}:&\quad t_{56}=t_{78},\qquad t_{57}=t_{68}.
\end{aligned}
\]
Combining these relations, we obtain
\[
V(R^*)
=
\left\{
(t_E)_{E\in\mathcal E(P)}
\ \middle|\
\begin{array}{l}
t_{12}=t_{34}=t_{56}=t_{78}=\beta,\\
t_E=\alpha\text{ for every other edge }E
\end{array}
\right\}.
\]
Thus \(\dim_{\CC}V(R^*)_{\CC}=2\). The homothety direction
\(\mathbf 1=(1)_{E\in\mathcal E(P)}\) is given by
\(\alpha=\beta\). Consequently, Theorem~\ref{th Altmann T1} gives
\(T^1_{X_P}(-R^*)\cong V(R^*)_{\CC}/\CC\mathbf 1\cong\CC\).
The nontrivial class is represented by the Minkowski decomposition
\(P=[0,e_1]+\operatorname{conv}\{0,e_2,e_3,e_1+e_2+e_3\}\).

Since \(P\) is full-dimensional, the restriction map from global
functionals to the direct sum of the vertex summands is injective and
has four-dimensional image. Moreover,
\(T^1_{X_P}(-R^*)\cong
\ker(g_1)/\operatorname{im}(\widetilde M_{\CC}^*)\).
It follows that \(\dim_{\CC}\ker(g_1)=4+1=5\) and
\(\operatorname{rank}(g_1)=24-5=19\).

The map \(g_2\) is surjective. Indeed, suppose that an element of the
dual of the last term lies in \(\ker(g_2^*)\). If two facets \(F\) and
\(F'\) meet along an edge \(E\), then \(\widehat F^\perp\) and
\(\widehat{F'}^\perp\) are distinct one-dimensional subspaces of the
two-dimensional space \(\widehat E^\perp\). The \(E\)-component of its
image under \(g_2^*\) can therefore vanish only if its coefficients at
both \(F\) and \(F'\) are zero. Hence \(g_2^*\) is injective, so \(g_2\)
is surjective. Consequently, \(\operatorname{rank}(g_2)=8\) and
\(\dim_{\CC}\ker(g_2)=28-8=20\).

Finally,
$$\dim_{\CC}T^2_{X_P}(-R^*)=
\dim_{\CC}\ker(g_2)-\operatorname{rank}(g_1)=20-19=1.$$
Therefore \(T^2_{X_P}(-R^*)\cong\CC\).
\end{example}

\subsection{The tangent space}

{
\begin{proposition}
\label{prop empty T1 degrees}
Let \(P\subset N_{\RR}\) be an empty three-dimensional lattice
polytope, and let \(X_P\) be the associated affine Gorenstein toric
variety. For \(r\in\widetilde M\), set
\[
\varphi_r(v):=\langle r,(v,1)\rangle,
\qquad
F_r:=P\cap\{\varphi_r=1\}.
\]
For every facet \(F\subset P\), let \(s_F\in\widetilde M\) denote the
primitive generator of the corresponding ray of \(\sigma^\vee\), and
let \(\mathcal S(P)\) denote the set of square facets of \(P\).

Apart from the Gorenstein-degree component \(T^1_{X_P}(-R^*)\), the
nonzero homogeneous components of \(T^1_{X_P}\) are precisely those
associated with square facets. More explicitly, for \(r\neq R^*\),
\[
T^1_{X_P}(-r)\neq0
\quad\Longleftrightarrow\quad
r=R^*-ns_F
\]
for some \(F\in\mathcal S(P)\) and \(n\geq1\). In this case,
\(F_r=F\) and
\[
T^1_{X_P}\bigl(-(R^*-ns_F)\bigr)\cong\CC.
\]

Moreover, \(F_{R^*}=P\), and every \(r\) for which
\(T^1_{X_P}(-r)\neq0\) satisfies \(\varphi_r\leq1\) on \(P\).
Consequently,
\[
T^1_{X_P}
=
T^1_{X_P}(-R^*)
\oplus
\bigoplus_{F\in\mathcal S(P)}
\bigoplus_{n\geq1}
T^1_{X_P}\bigl(-(R^*-ns_F)\bigr).
\]
\end{proposition}
}

\begin{proof}
Since \(P\) is empty, every two-dimensional face of \(P\) is lattice
equivalent either to the standard triangle or to the unit square, by
Lemma~\ref{lem empty lattice polygons}. Theorem~\ref{th AvS comparison}
therefore shows that \(T^1_{X_P}(-r)\neq0\) only if
\(\varphi_r\leq1\) on \(P\), and that in this case either \(F_r=P\) or
\(F_r\) is a square facet.

If \(F_r=P\), then \(\varphi_r=1\) on the full-dimensional polytope
\(P\). Hence \(\varphi_r=\varphi_{R^*}\), and therefore \(r=R^*\).
Conversely, \(\varphi_{R^*}=1\) on \(P\), so \(F_{R^*}=P\).

Suppose now that \(r\neq R^*\) and \(T^1_{X_P}(-r)\neq0\). Then
\(F:=F_r\) is a square facet. Since both \(\varphi_r\) and
\(\varphi_{R^*}\) equal \(1\) on \(F\), the element \(R^*-r\) vanishes
on \(F\). Moreover, \(\varphi_r\leq1\) on \(P\), so \(R^*-r\) is
nonnegative on \(P\). It therefore belongs to the ray of
\(\sigma^\vee\) corresponding to \(F\). Since \(s_F\) is the primitive
lattice generator of this ray, \(R^*-r=ns_F\) for some \(n\geq1\), or
equivalently \(r=R^*-ns_F\).

Conversely, let \(F\) be a square facet and set \(r=R^*-ns_F\) for
some \(n\geq1\). Since \(\varphi_{s_F}\geq0\) on \(P\) and vanishes
precisely on \(F\), we have
\(\varphi_r=1-n\varphi_{s_F}\leq1\) on \(P\), with equality precisely
on \(F\). Thus \(F_r=F\).

For a unit square, \(V(r)_{\CC}\cong\CC^2\), with one coordinate for
each pair of opposite edges. Theorem~\ref{th Altmann T1} therefore gives
\[
T^1_{X_P}(-r)
\cong
V(r)_{\CC}/\CC\mathbf 1
\cong
\CC.
\]
This proves the classification of the nonzero homogeneous components,
and the asserted direct-sum decomposition follows.
\end{proof}

{
The following example applies
Proposition~\ref{prop empty T1 degrees} to the eight-vertex polytope
and records all nonzero homogeneous components of its tangent space.

\begin{example}\label{ex eight vertex T1 degrees}
Continue with the polytope \(P\) of
Example~\ref{ex eight vertex polytope hilbert basis}. Retain the
notation \(F_{y,0},F_{y,1},F_{z,0},F_{z,1}\) for its four square
facets. Their primitive support elements are
\(s_{F_{y,0}}=s_2\), \(s_{F_{y,1}}=s_7\),
\(s_{F_{z,0}}=s_3\), and \(s_{F_{z,1}}=s_8\).

By Proposition~\ref{prop empty T1 degrees}, for every \(n\geq1\) and
\(i\in\{2,3,7,8\}\), we have
\(T^1_{X_P}(-(R^*-ns_i))\cong\CC\). Moreover,
Example~\ref{ex eight vertex T2} showed that
\(T^1_{X_P}(-R^*)\cong\CC\). Hence
\[
T^1_{X_P}
=
T^1_{X_P}(-R^*)
\oplus
\bigoplus_{i\in\{2,3,7,8\}}
\bigoplus_{n\geq1}
T^1_{X_P}\bigl(-(R^*-ns_i)\bigr),
\]
and every homogeneous component appearing in this decomposition is
one-dimensional.
\end{example}
}

\subsection{The obstruction space}

\begin{proposition}\label{th gen}
Let \(P\subset N_{\RR}\), where \(N\simeq\ZZ^3\), be an empty
three-dimensional lattice polytope, and let \(X_P\) be the associated
four-dimensional affine Gorenstein toric variety. Then
$
T^2_{X_P}(-r)=0
$
for every \(r\in\widetilde M\) with \(r\neq R^*\).
\end{proposition}

\begin{proof}
If there exists a vertex \(v\) of \(P\) such that
\(\varphi_r(v)>1\), then the assertion follows from
Theorem~\ref{th AvS comparison}(1). Suppose therefore that
\(\varphi_r\leq1\) on \(P\). Since \(r\neq R^*\), the section
$
F_r=P\cap(\varphi_r=1)
$
is a proper face of \(P\), and hence \(\dim F_r\leq2\). The assertion
now follows from Theorem~\ref{th AvS comparison}(3).
\end{proof}

\section{Unobstructedness of affine terminal Gorenstein toric fourfolds}
\label{sec 5 unobstructedness}

In this section, we prove that every affine terminal Gorenstein toric
fourfold is unobstructed.

{
Throughout this section, let \(P\subset N_{\RR}\) be an empty
three-dimensional lattice polytope and let \(X_P\) be the associated
affine toric fourfold. We denote by \(\mathcal S(P)\) the set of square
facets of \(P\). For \(F\in\mathcal S(P)\), let
\(s_F\in\widetilde M\) be the primitive generator of the ray of
\(\sigma^\vee\) corresponding to \(F\), and set
\(m_F:=R^*-s_F\).}

By Proposition~\ref{th gen}, the only homogeneous component of
\(T^2_{X_P}\) which can be nonzero is \(T^2_{X_P}(-R^*)\).
On the other hand, Proposition~\ref{prop empty T1 degrees} gives
\(T^1_{X_P}(-m_F)\cong\CC\) for every \(F\in\mathcal S(P)\).
Choose a tangent parameter \(t_F\) corresponding to this component and
assign it degree \(\deg(t_F)=m_F\).

A monomial
\[
\prod_{F\in\mathcal S(P)}t_F^{n_F},
\qquad n_F\in\NN,
\]
can contribute to a nonzero obstruction only if its degree is \(R^*\),
that is, only if
\[
\sum_{F\in\mathcal S(P)}n_Fm_F=R^*.
\]
The following lemma classifies all such relations. It shows that the
only possibility involves two parallel square facets, each occurring
with coefficient one.

{
\begin{lemma}\label{lem square relation}
With the notation above, suppose that
\[
\sum_{F\in\mathcal S(P)}n_F(R^*-s_F)=R^*
\]
for some \(n_F\in\NN\). Then there exist precisely two square facets
\(F_1,F_2\in\mathcal S(P)\) such that
\(n_{F_1}=n_{F_2}=1\), while \(n_F=0\) for every
\(F\notin\{F_1,F_2\}\). Moreover, \(F_1\) and \(F_2\) are parallel.
\end{lemma}
}

\begin{proof}
We first show that, for every square facet \(F\) and every vertex
\(v\in P\setminus F\),
\begin{equation}\label{eq square height one}
\langle s_F,(v,1)\rangle=1.
\end{equation}

After an affine lattice change of coordinates, we may assume that
\begin{equation}
\label{height-1-section}
F
=
\operatorname{conv}
\{(0,0,0),(1,0,0),(0,1,0),(1,1,0)\}
\end{equation}
and that
$
\langle s_F,(x,y,z,1)\rangle=z.
$
Let
$
v=(a,b,h)
$
be a vertex of \(P\) not contained in \(F\). Then \(h\geq1\). Suppose
that \(h\geq2\). The section of the pyramid
\(\operatorname{conv}(F,v)\) at height \(1\) is
\[
\operatorname{conv}(F,v)\cap\{z=1\}
=
\left(
\frac{a}{h},
\frac{b}{h},
1
\right)
+
\left(1-\frac1h\right)
([0,1]^2\times\{0\}).
\]
Indeed, every point of \(\operatorname{conv}(F,v)\) can be written as
\((1-\lambda)w+\lambda v\), where \(w=(x,y,0)\in F\) and
\(0\leq\lambda\leq1\). Its \(z\)-coordinate is \(\lambda h\), so a
point lies in the section \(z=1\) precisely when
\(\lambda=1/h\). Consequently, the points of this section are exactly
\[
\left(
\frac{a}{h}+\left(1-\frac1h\right)x,\,
\frac{b}{h}+\left(1-\frac1h\right)y,\,
1
\right),
\qquad (x,y)\in[0,1]^2,
\]
which verifies \eqref{height-1-section}.

We next verify that this section contains a lattice point. Write
\(a=dh+r\), where \(0\leq r<h\). If \(r=0\), then
\(\lceil a/h\rceil=a/h\). If \(r>0\), then
\(\lceil a/h\rceil-a/h=(h-r)/h\leq(h-1)/h=1-1/h\). Thus
\(\lceil a/h\rceil\in[a/h,a/h+1-1/h]\), and the same argument applies
to \(b\). Hence
\[
q:=
\left(
\left\lceil\frac ah\right\rceil,\,
\left\lceil\frac bh\right\rceil,\,
1
\right)
\]
belongs to \(\operatorname{conv}(F,v)\cap\{z=1\}\) and is a lattice
point.

Since \(0<1<h\), the point \(q\) is neither a vertex of \(F\) nor the
vertex \(v\). Moreover, it is a nontrivial convex combination of \(v\)
and a point of \(F\), so it is not a vertex of \(P\). This contradicts
the assumption that \(P\) is empty. Therefore \(h=1\), proving
\eqref{eq square height one}.

Consequently, for every vertex \(v\) of \(P\),
\[
\langle R^*-s_F,(v,1)\rangle
=
\begin{cases}
1, & v\in F,\\
0, & v\notin F.
\end{cases}
\]

Pairing
\[
\sum_{F\in\mathcal S(P)}
n_F(R^*-s_F)=R^*
\]
with \((v,1)\), for an arbitrary vertex \(v\) of \(P\), gives
\begin{equation}\label{eq vertex covering}
\sum_{\substack{F\in\mathcal S(P)\\ v\in F}}n_F=1.
\end{equation}
Set
$
\mathcal S'
:=
\{F\in\mathcal S(P)\mid n_F>0\}.
$
Equation~\eqref{eq vertex covering} shows that \(n_F=1\) for every
\(F\in\mathcal S'\), that distinct facets in \(\mathcal S'\) are
vertex-disjoint, and that every vertex of \(P\) belongs to exactly one
such facet. Since every facet in \(\mathcal S'\) has four vertices and
\(P\) has at most eight vertices by
Lemma~\ref{lem empty polytope vertices}, we have
$
|\mathcal S'|\leq2.
$
Moreover, \(|\mathcal S'|\neq1\), since otherwise all vertices of \(P\)
would lie in a single two-dimensional facet. Hence
\[
\mathcal S'=\{F_1,F_2\},
\qquad
n_{F_1}=n_{F_2}=1,
\]
and all the remaining coefficients vanish.

The original relation becomes
\[
(R^*-s_{F_1})+(R^*-s_{F_2})=R^*,
\]
or equivalently
$
s_{F_1}+s_{F_2}=R^*.
$
Applying the projection
$
\pi_M\colon\widetilde M\longrightarrow M
$
gives
\[
\pi_M(s_{F_1})+\pi_M(s_{F_2})=0.
\]
Thus the primitive normal vectors of \(F_1\) and \(F_2\) are opposite,
and hence \(F_1\) and \(F_2\) are parallel.
\end{proof}

\begin{example}\label{ex eight vertex primitive sums}
Let \(P\) be the polytope of
Example~\ref{ex eight vertex polytope hilbert basis}. For its four
square facets, the corresponding primitive tangent degrees are
\[
\begin{aligned}
m_{y,0}&=R^*-s_{F_{y,0}}=R^*-s_2=s_7,
&\hspace{3em}
m_{y,1}&=R^*-s_{F_{y,1}}=R^*-s_7=s_2,\\
m_{z,0}&=R^*-s_{F_{z,0}}=R^*-s_3=s_8,
&\hspace{3em}
m_{z,1}&=R^*-s_{F_{z,1}}=R^*-s_8=s_3.
\end{aligned}
\]
Consequently, the two pairs of opposite square facets give the
relations
\[
m_{y,0}+m_{y,1}
=
s_7+s_2
=
R^*
\qquad
\text{and}
\qquad
m_{z,0}+m_{z,1}
=
s_8+s_3
=
R^*.
\]
In Example~\ref{ex eight vertex two parameter deformations}, we
constructed a two-parameter deformation associated with each of these
two relations. Thus this polytope admits two distinct two-parameter
deformation families arising from its two pairs of opposite square
facets.
\end{example}

\begin{theorem}
\label{th term un}
\label{cor terminal Gorenstein unobstructed}
Let \(X_P\) be an affine terminal Gorenstein toric variety with
$
\dim X_P\leq4.
$
Then \(X_P\) is unobstructed.
\end{theorem}

\begin{proof}
The assertion for \(\dim X_P\leq3\) follows from
Lemma~\ref{lem for terminal dim 3}. We therefore assume that
$
\dim X_P=4,
$
so that \(P\) is an empty three-dimensional lattice polytope.

{
By Proposition~\ref{prop empty T1 degrees}, for each square facet
\(F\in\mathcal S(P)\), the corresponding nonzero tangent components
occur in degrees \(-(R^*-ns_F)\), with \(n\geq 1\). We call the
component corresponding to \(n=1\) the primitive tangent component
associated with \(F\). 

We first restrict to the finite-dimensional primitive tangent
subspace
\[
V_{\mathrm{prim}}
:=
T^1_{X_P}(-R^*)
\oplus
\bigoplus_{F\in\mathcal S(P)}
T^1_{X_P}\bigl(-(R^*-s_F)\bigr).
\]
Thus \(V_{\mathrm{prim}}\) consists of the Gorenstein-degree component
together with the primitive tangent component associated with each
square facet of \(P\).
}

For every \(F\in\mathcal S(P)\), put
$
m_F:=R^*-s_F.
$
By \eqref{eq square height one},
\[
\langle m_F,(v,1)\rangle
=
\begin{cases}
1, & v\in F,\\
0, & v\notin F.
\end{cases}
\]
Consequently,
$
m_F\in\sigma^\vee\setminus\{0\}.
$

{By Lemma~\ref{lem empty lattice polygons}, the facet \(F\) is lattice
equivalent to the unit square. Hence there exist a vertex \(p\in F\)
and primitive lattice vectors \(e_1,e_2\in N\) such that
\[
F=p+[0,e_1]+[0,e_2].
\]
Since \(e_1\) and \(e_2\) are parallel to \(F\), they are annihilated
by \(\pi_M(s_F)\), and therefore also by
\(\pi_M(m_F)=-\pi_M(s_F)\). Thus
\(Q_F:=[0,e_1]\) is a primitive lattice segment contained in
\(\ker\pi_M(m_F)\), parallel to an edge of \(F\), and a Minkowski
summand of \(F\).}
Then \((m_F,Q_F)\) is a deformation pair.  Let \(t_F\) denote
the parameter of the corresponding one-parameter deformation and set
\[
\xi_F
:=
\KS\left(\frac{\partial}{\partial t_F}\right)
\in
T^1_{X_P}(-m_F).
\]
{
By restricting \eqref{eq two summand Fk} to the
\(t_F\)-axis, its first-order part is
\[
F_{\mathbf k}
=
f_{\mathbf k}
-
\eta_{Q_F}(\mathbf k)t_F
\chi^{s_{\mathbf k}-m_F}
+
O(t_F^2).
\]
Hence \(\xi_F\) is the Minkowski-summand class determined by \(Q_F\).
Under the identification of Theorem~\ref{th Altmann T1}, this class is
nonzero because \(Q_F\) is a nontrivial Minkowski summand of the unit
square \(F\). Since \(T^1_{X_P}(-m_F)\) is one-dimensional,
\(\xi_F\) is a generator.}


Let
$
d:=\dim_{\CC}T^1_{X_P}(-R^*)
$
and choose a basis
$
\zeta_1,\ldots,\zeta_d
$
of \(T^1_{X_P}(-R^*)\). Introduce variables
$
z_1,\ldots,z_d
$
of degree \(R^*\), and put
\[
R_{\mathrm{prim}}
:=
\CC\Bigl[\!\Bigl[
z_1,\ldots,z_d,
(t_F)_{F\in\mathcal S(P)}
\Bigr]\!\Bigr],
\]
where
\[
\deg(z_i)=R^*,
\qquad
\deg(t_F)=m_F=R^*-s_F.
\]
Let
\[
\mathfrak m_{\mathrm{prim}}
:=
\bigl(z_1,\ldots,z_d,(t_F)_{F\in\mathcal S(P)}\bigr)
\]
and, for \(k\geq1\), set
\[
R_{\mathrm{prim}}^{(k)}
:=
R_{\mathrm{prim}}/
\mathfrak m_{\mathrm{prim}}^{k+1}.
\]

Since \(T^1_{X_P}\) may be infinite-dimensional, no complete local
noetherian base can carry independent parameters for all its tangent
directions. We therefore begin with the finite-dimensional subspace
\(V_{\mathrm{prim}}\) and construct a formal deformation in all orders.

More precisely, for each \(k\geq1\), we construct inductively a graded
\(k\)-th order deformation
\[
\mathcal X_{\mathrm{prim}}^{(k)}
\longrightarrow
\operatorname{Spec}R_{\mathrm{prim}}^{(k)}
\]
of \(X_P\). These deformations are compatible in the sense that the
pullback of \(\mathcal X_{\mathrm{prim}}^{(k+1)}\) to
\(\operatorname{Spec}R_{\mathrm{prim}}^{(k)}\) is isomorphic to
\(\mathcal X_{\mathrm{prim}}^{(k)}\). Their Kodaira--Spencer maps
identify the coordinate tangent vectors with
\[
\zeta_1,\ldots,\zeta_d,
\qquad
(\xi_F)_{F\in\mathcal S(P)}.
\]

The total degree-zero element
\[
\sum_{i=1}^d\zeta_i\otimes z_i
+
\sum_{F\in\mathcal S(P)}\xi_F\otimes t_F
\in
T^1_{X_P}\otimes_{\CC}
\mathfrak m_{\mathrm{prim}}/
\mathfrak m_{\mathrm{prim}}^2
\]
determines a graded first-order deformation
$
\mathcal X_{\mathrm{prim}}^{(1)}
\longrightarrow
\operatorname{Spec}R_{\mathrm{prim}}^{(1)}.
$

Suppose inductively that
\(\mathcal X_{\mathrm{prim}}^{(k)}\) has been constructed.
The obstruction
to lifting it to \(R_{\mathrm{prim}}^{(k+1)}\) is a sum of terms
\[
\omega_{\boldsymbol\alpha,\mathbf a}
\otimes
z_1^{\alpha_1}\cdots z_d^{\alpha_d}
\prod_{F\in\mathcal S(P)}t_F^{a_F},
\]
where
\[
\omega_{\boldsymbol\alpha,\mathbf a}
\in
T^2_{X_P}\left(
-
\left(
|\boldsymbol\alpha|R^*
+
\sum_{F\in\mathcal S(P)}a_Fm_F
\right)
\right)
\]
and
\[
|\boldsymbol\alpha|
+
\sum_Fa_F
=
k+1.
\]

By Proposition~\ref{th gen}, this coefficient can be nonzero only if
\begin{equation}
\label{eq primitive obstruction degree}
|\boldsymbol\alpha|R^*
+
\sum_{F\in\mathcal S(P)}a_Fm_F
=
R^*.
\end{equation}

Suppose first that
$
|\boldsymbol\alpha|\geq1.
$
{Equation
\eqref{eq primitive obstruction degree} is equivalent to} 
\[
\bigl(|\boldsymbol\alpha|-1\bigr)R^*
+
\sum_{F\in\mathcal S(P)}a_Fm_F
=
0.
\]
All the terms in this sum belong to the strongly convex cone
\(\sigma^\vee\). Hence every term is zero. Since
$R^*\neq0$ and
$m_F\neq0$,
we obtain
$
|\boldsymbol\alpha|=1$
and
$a_F=0$
for every \(F\). This gives a monomial of ordinary order one, whereas
an obstruction to lifting
\(R_{\mathrm{prim}}^{(k)}\) to
\(R_{\mathrm{prim}}^{(k+1)}\) has ordinary order \(k+1\geq2\).
Consequently, no obstruction term can involve one of the variables
\(z_i\).

It remains to consider monomials of the form
\[
\prod_{F\in\mathcal S(P)}t_F^{a_F}.
\]
Their degree is \(R^*\) precisely when
\[
\sum_{F\in\mathcal S(P)}
a_F(R^*-s_F)
=
R^*.
\]
By Lemma~\ref{lem square relation}, this can occur only when there are
two parallel square facets \(F_1,F_2\) such that
\[
a_{F_1}=a_{F_2}=1,
\qquad
a_F=0
\quad
\text{for }F\notin\{F_1,F_2\},
\]
and
$
s_{F_1}+s_{F_2}=R^*.
$
Thus the only possible obstruction monomial is
$
t_{F_1}t_{F_2}.
$
In particular, such an obstruction can occur only in the passage from
first to second order.

Let
$
\omega_{F_1,F_2}
\in T^2_{X_P}(-R^*)
$
denote the coefficient of this monomial in the second-order
obstruction. Put
\[
B:=\CC[[u_1,u_2]],
\qquad
\mathfrak n:=(u_1,u_2),
\qquad
B^{(k)}:=B/\mathfrak n^{k+1}.
\]
Consider the graded homomorphism
$
R_{\mathrm{prim}}^{(1)}
\longrightarrow
B^{(1)}
$
given by
\[
t_{F_1}\longmapsto u_1,
\qquad
t_{F_2}\longmapsto u_2,
\]
and by sending all the remaining variables to zero.

By Corollary~\ref{cor parallel squares simultaneous deformation},
there is a formal two-parameter deformation over \(B\) whose
Kodaira--Spencer classes in the \(u_1\)- and \(u_2\)-directions are
\(\xi_{F_1}\) and \(\xi_{F_2}\), respectively. Its first-order
truncation therefore represents the pullback of
\(\mathcal X_{\mathrm{prim}}^{(1)}\), while its second-order
truncation provides a lifting of that first-order deformation to
\(B^{(2)}\).

By functoriality of obstruction classes, the pullback of the
second-order obstruction is consequently zero. On the other hand,
this pullback is
\[
\omega_{F_1,F_2}\otimes u_1u_2
\in
T^2_{X_P}(-R^*)\otimes_{\CC}
\mathfrak n^2/\mathfrak n^3.
\]
Since
\[
u_1u_2\neq0
\qquad\text{in}\qquad
\mathfrak n^2/\mathfrak n^3,
\]
it follows that
$
\omega_{F_1,F_2}=0.
$

Equivalently, after choosing compatible parameter coordinates, the
comparison homomorphism has the form
\[
t_{F_i}\longmapsto u_i+O(\mathfrak n^2),
\qquad i=1,2.
\]
Therefore
\[
t_{F_1}t_{F_2}
\longmapsto
u_1u_2
\pmod{\mathfrak n^3},
\]
so a higher-order change of parameter cannot alter the coefficient of
the mixed quadratic term.

We have proved that every second-order obstruction vanishes. For
orders at least three, the degree calculation above shows that there
is no monomial whose coefficient could lie in a nonzero homogeneous
component of \(T^2_{X_P}\). We may therefore construct inductively
compatible graded liftings
\[
\mathcal X_{\mathrm{prim}}^{(k+1)}
\longrightarrow
\operatorname{Spec}R_{\mathrm{prim}}^{(k+1)}
\]
for all \(k\geq1\). 
{For a complete local ring \((R,\mathfrak m)\), we write
\(\operatorname{Spf}R\) for its formal spectrum, namely the affine
formal scheme determined by the infinitesimal thickenings
\(\operatorname{Spec}(R/\mathfrak m^{k+1})\), \(k\geq0\).}
Passing to the inverse limit gives a formal
deformation
$
\mathcal X_{\mathrm{prim}}
\longrightarrow
\operatorname{Spf}R_{\mathrm{prim}}
$
whose Kodaira--Spencer map is an isomorphism onto
\[
V_{\mathrm{prim}}
=
T^1_{X_P}(-R^*)
\oplus
\bigoplus_{F\in\mathcal S(P)}
T^1_{X_P}\bigl(-(R^*-s_F)\bigr).
\]

We next incorporate the remaining homogeneous tangent directions.
For every square facet \(F\), let
$
x_F\in\CC[S_P]
$
denote the character \(\chi^{s_F}\). For \(n\geq1\), put
$
r_{F,n}:=R^*-ns_F.
$
By \eqref{eq square height one},
\[
\langle r_{F,n},(v,1)\rangle
=
\begin{cases}
1,   & v\in F,\\
1-n, & v\notin F.
\end{cases}
\]
It follows directly from the definition of the sets \(K_\tau^r\)
that
$
K_\tau^{r_{F,n}}
=
K_\tau^{r_{F,1}}
$
for every face \(\tau\preceq\sigma\). Indeed, at a vertex of \(F\)
the upper bound is one for every \(n\), while at a vertex outside
\(F\) it is nonpositive, so the corresponding \(K\)-set is empty.

By \cite[Theorem~3.3]{AS98}, multiplication by \(x_F^{n-1}\) is
induced on the complexes computing \(T^1_{X_P}\) by the inclusions
$
K_\tau^{r_{F,n}}
\subseteq
K_\tau^{r_{F,1}}.
$
Since these inclusions are equalities, multiplication induces an
isomorphism
\[
x_F^{n-1}\cdot-
\colon
T^1_{X_P}\bigl(-(R^*-s_F)\bigr)
\xrightarrow{\ \sim\ }
T^1_{X_P}\bigl(-(R^*-ns_F)\bigr).
\]
Consequently,
\begin{equation}
\label{eq higher tangent generator}
T^1_{X_P}\bigl(-(R^*-ns_F)\bigr)
=
\CC\,x_F^{n-1}\xi_F.
\end{equation}

We shall use the following flatness lemma for the coordinate-dependent
substitution below.

\begin{lemma}\label{lem coordinate dependent substitution}
Let \((B,\mathfrak m_B)\) and \((C,\mathfrak m_C)\) be complete local
Noetherian \(\CC\)-algebras with residue field \(\CC\), let
\(\mathcal P_0\) be a polynomial \(\CC\)-algebra, and put
\[
\mathcal P_B:=B\widehat\otimes_{\CC}\mathcal P_0,
\qquad
\mathcal P_C:=C\widehat\otimes_{\CC}\mathcal P_0.
\]
Suppose that
\[
\mathcal A_B
=
\mathcal P_B/(G_1,\ldots,G_e)
\]
is \(B\)-flat. Then one can choose finitely many syzygies
\[
V_1,\ldots,V_b\in\mathcal P_B^{\oplus e},
\qquad
(G_1,\ldots,G_e)V_j=0,
\]
whose reductions modulo \(\mathfrak m_B\) generate all syzygies among
the reduced equations \(g_1,\ldots,g_e\).

Let
\(\Phi\colon\mathcal P_B\to\mathcal P_C\) be a continuous
\(\CC\)-algebra homomorphism such that
\[
\Phi(\mathfrak m_B\mathcal P_B)
\subseteq
\mathfrak m_C\mathcal P_C
\]
and whose reduction modulo the maximal ideals is the identity on
\(\mathcal P_0\). Put
\[
G_i':=\Phi(G_i),
\qquad
V_j':=\Phi(V_j).
\]
If \((G_1',\ldots,G_e')V_j'=0\) for every \(j\), then
\[
\mathcal P_C/(G_1',\ldots,G_e')
\]
is \(C\)-flat.
\end{lemma}

\begin{proof}
Set \(I_B:=(G_1,\ldots,G_e)\subset P_B\), and let
\(K_B:=\ker(P_B^{\oplus e}\to I_B)\), where the \(i\)-th basis vector
maps to \(G_i\). By construction, \(P_B\) is flat over \(B\), while
\(\mathcal A_B=P_B/I_B\) is flat over \(B\) by assumption. The exact sequence
\[
0\longrightarrow I_B\longrightarrow P_B\longrightarrow \mathcal A_B
\longrightarrow0
\]
therefore implies that \(I_B\) is flat over \(B\).

Tensoring this sequence and
\[
0\longrightarrow K_B\longrightarrow P_B^{\oplus e}
\longrightarrow I_B\longrightarrow0
\]
with \(B/\mathfrak m_B\cong\CC\), and using the flatness of \(\mathcal A_B\)
and \(I_B\), gives
\[
K_B/\mathfrak m_BK_B
\cong
\ker\bigl(P_0^{\oplus e}\longrightarrow I_0\bigr),
\qquad
I_0:=(g_1,\ldots,g_e).
\]
Since \(P_0\) is noetherian, the module on the right is finitely
generated. Choose generators and lift them to syzygies
\(V_1,\ldots,V_b\in K_B\). Their reductions modulo
\(\mathfrak m_B\) generate all syzygies among \(g_1,\ldots,g_e\).
If the data are graded, the generators and their lifts may be chosen
homogeneous.

Put \(\mathcal A_C:=P_C/(G_1',\ldots,G_e')\). Since
\(G_i'=\Phi(G_i)\) and \(V_j'=\Phi(V_j)\), we have
\((G_1',\ldots,G_e')V_j'=0\). Moreover, because \(\Phi\) reduces to
the identity on \(P_0\), the reductions of the \(G_i'\) and \(V_j'\)
are \(g_i\) and \(V_j\bmod\mathfrak m_B\), respectively.

For \(q\geq0\), put \(C_q:=C/\mathfrak m_C^{q+1}\) and
\(P_q:=P_C/\mathfrak m_C^{q+1}P_C\), and let \(I_q\subset P_q\) be
the ideal generated by the images of \(G_1',\ldots,G_e'\). We claim
that \(A_q:=P_q/I_q\) is flat over \(C_q\).

Let \(h\in I_q\cap\mathfrak m_CP_q\), and write
\(h=(G_1',\ldots,G_e')c\) for some \(c\in P_q^{\oplus e}\). Reducing
modulo \(\mathfrak m_C\), we obtain the syzygy
\((g_1,\ldots,g_e)\overline c=0\). Hence
\(\overline c=\sum_{j=1}^b\overline d_j\,\overline V_j\) for suitable
\(\overline d_j\in P_0\). After choosing lifts \(d_j\in P_q\), we have
\(c-\sum_jd_jV_j'\in\mathfrak m_CP_q^{\oplus e}\). Since
\((G_1',\ldots,G_e')V_j'=0\), it follows that
\(h\in\mathfrak m_CI_q\). Therefore
\[
I_q\cap\mathfrak m_CP_q=\mathfrak m_CI_q.
\]
Since \(A_q/\mathfrak m_CA_q\cong P_0/I_0\) is a \(\CC\)-vector space,
the local criterion for flatness over the Artinian local ring \(C_q\)
shows that \(A_q\) is flat over \(C_q\). This is the presentation form
of the equational flatness criterion; see, e.g.,
\cite[Lemma~3.5]{Fil26}.

Finally, \(P_C\) is noetherian and \(\mathfrak m_C\)-adically
complete. Hence the finitely generated ideal
\((G_1',\ldots,G_e')\) is closed, \(\mathcal A_C\) is
\(\mathfrak m_C\)-adically complete, and
\[
\mathcal A_C\cong\varprojlim_q A_q.
\]
The transition maps \(A_{q+1}\to A_q\) are surjective, and every
\(A_q\) is flat over \(C_q\). The inverse-limit criterion for
flatness over a noetherian ring therefore implies that \(\mathcal A_C\) is
flat over \(C\).
\end{proof}

We now perform the substitution at the level of equations and
relations. Let
\[
\mathcal P_0:=\CC[\bfx,u],
\qquad
A_0:=\CC[S_P]=\mathcal P_0/I_P,
\]
and put
\[
\mathcal P_{\mathrm{prim}}
:=
R_{\mathrm{prim}}
\widehat\otimes_{\CC}\mathcal P_0.
\]
After lifting the ambient coordinate functions to the primitive family,
write
\[
\mathcal A_{\mathrm{prim}}
=
\mathcal P_{\mathrm{prim}}/I_{\mathrm{prim}}.
\]
The ring \(\mathcal P_{\mathrm{prim}}\) is Noetherian, so we may choose
finite homogeneous generators
\[
G_1,\ldots,G_e
\]
of \(I_{\mathrm{prim}}\). Since
\(\mathcal A_{\mathrm{prim}}\) is flat over \(R_{\mathrm{prim}}\), their
reductions \(g_1,\ldots,g_e\) generate \(I_P\). Choose finitely many
homogeneous syzygies among the \(G_i\) as in
Lemma~\ref{lem coordinate dependent substitution}, and let
\(v_{\mathrm{prim}}\) be the matrix whose columns are these syzygies.
Let \(u_{\mathrm{prim}}\) be the row matrix with components \(G_i\), and
let \(u_0,v_0\) denote their reductions modulo
\(\mathfrak m_{\mathrm{prim}}\). We then have a homogeneous complex
\begin{equation}
\label{eq primitive finite presentation}
\mathcal P_{\mathrm{prim}}^{\oplus b}
\xrightarrow{\,v_{\mathrm{prim}}\,}
\mathcal P_{\mathrm{prim}}^{\oplus e}
\xrightarrow{\,u_{\mathrm{prim}}\,}
\mathcal P_{\mathrm{prim}}
\longrightarrow
\mathcal A_{\mathrm{prim}}
\longrightarrow0,
\end{equation}
with
$
u_{\mathrm{prim}}\circ v_{\mathrm{prim}}=0,
$
whose reduction is the finite homogeneous presentation
\begin{equation}
\label{eq central finite presentation}
\mathcal P_0^{\oplus b}
\xrightarrow{\,v_0\,}
\mathcal P_0^{\oplus e}
\xrightarrow{\,u_0\,}
\mathcal P_0
\longrightarrow
A_0
\longrightarrow0.
\end{equation}
In particular, the columns of \(v_0\) generate all relations among the
central equations \(g_i\).

For every \(F\in\mathcal S(P)\), choose a homogeneous polynomial, still denoted by
\(x_F\), in \(\mathcal P_0\) whose image in \(A_0\) is
\(\chi^{s_F}\). For \(N\geq1\), introduce parameters
\[
t_{F,n},
\qquad
F\in\mathcal S(P),
\quad
1\leq n\leq N,
\]
of degree
$
\deg(t_{F,n})=R^*-ns_F,
$
and put
\[
R_N
:=
\CC\Bigl[\!\Bigl[
z_1,\ldots,z_d,
(t_{F,n})_{\substack{F\in\mathcal S(P)\\1\leq n\leq N}}
\Bigr]\!\Bigr].
\]
Set
\[
\widehat t_F^{(N)}
:=
\sum_{n=1}^N x_F^{n-1}t_{F,n}.
\]
Every summand is homogeneous of degree
\[
\deg\bigl(x_F^{n-1}t_{F,n}\bigr)
=
(n-1)s_F+(R^*-ns_F)
=
R^*-s_F
=
\deg(t_F).
\]

Let
$
\mathcal P_N
:=
R_N\widehat\otimes_{\CC}\mathcal P_0.
$
There is a continuous homogeneous \(\CC\)-algebra homomorphism
\[
\Phi_N\colon
\mathcal P_{\mathrm{prim}}
\longrightarrow
\mathcal P_N
\]
which is the identity on \(\mathcal P_0\) and on the variables \(z_i\),
and satisfies
$
\Phi_N(t_F)=\widehat t_F^{(N)}.
$
Notice that \(\Phi_N\) is not asserted to be a base change
homomorphism: the coefficients \(x_F^{n-1}\) belong to the ambient
coordinate ring.

Apply \(\Phi_N\) to both matrices in
\eqref{eq primitive finite presentation}, and set
\[
u_N:=\Phi_N(u_{\mathrm{prim}}),
\qquad
v_N:=\Phi_N(v_{\mathrm{prim}}).
\]
Then
\[
u_N\circ v_N
=
\Phi_N(u_{\mathrm{prim}}\circ v_{\mathrm{prim}})
=
0.
\]
Moreover, modulo the maximal ideal of \(R_N\), the resulting complex
reduces to \eqref{eq central finite presentation}. Thus
\[
\mathcal P_N^{\oplus b}
\xrightarrow{\,v_N\,}
\mathcal P_N^{\oplus e}
\xrightarrow{\,u_N\,}
\mathcal P_N
\longrightarrow
\mathcal A_N
\longrightarrow0,
\qquad
\mathcal A_N:=\mathcal P_N/\operatorname{im}(u_N),
\]
is a complex which is exact except possibly at
\(\mathcal P_N^{\oplus e}\), and whose reduction over the closed point
is exact.

The homomorphism \(\Phi_N\) preserves the parameter ideals and induces
the identity on the closed fibre. The equality
\(u_N\circ v_N=0\) shows that all substituted syzygies still vanish,
and their reductions generate all syzygies among the central equations.
Lemma~\ref{lem coordinate dependent substitution} therefore implies
that \(\mathcal A_N\) is flat over \(R_N\). Hence it defines a formal
deformation
$
\mathcal X_N\longrightarrow\operatorname{Spf}R_N
$
of \(X_P\). 
Writing
$
G_i^{(N)}:=\Phi_N(G_i),
$
we have
\[
\left.
\frac{\partial G_i^{(N)}}{\partial t_{F,n}}
\right|_{\mathbf z=\mathbf t=0}
=
x_F^{n-1}
\left.
\frac{\partial G_i}{\partial t_F}
\right|_{\mathbf z=\mathbf t=0}.
\]
By the \(\CC[S_P]\)-module structure on \(T^1_{X_P}\), this gives
\[
\KS\left(\frac{\partial}{\partial t_{F,n}}\right)
=
x_F^{n-1}\xi_F.
\]
By \eqref{eq higher tangent generator}, this is a generator of
$
T^1_{X_P}\bigl(-(R^*-ns_F)\bigr).
$
The \(z_i\)-directions map to the chosen basis
\(\zeta_1,\ldots,\zeta_d\). Consequently, the Kodaira--Spencer map of
\(\mathcal X_N\) is an isomorphism onto
\[
V_N
:=
T^1_{X_P}(-R^*)
\oplus
\bigoplus_{F\in\mathcal S(P)}
\bigoplus_{1\leq n\leq N}
T^1_{X_P}\bigl(-(R^*-ns_F)\bigr).
\]

The families \(\mathcal X_N\) are compatible as \(N\) varies:
\(\mathcal X_{N+1}\) restricts to \(\mathcal X_N\) after setting
$
t_{F,N+1}=0
$
for every \(F\). Furthermore,
\[
T^1_{X_P}
=
\bigcup_{N\geq1}V_N.
\]

We claim that every deformation of \(X_P\) over a local Artinian
\(\CC\)-algebra \(B\) is induced from \(\mathcal X_N\) for some
\(N\).

We prove this by induction on the length of \(B\). The assertion is
clear for \(B=\CC\). Choose a small extension
\[
0\longrightarrow J
\longrightarrow B
\longrightarrow\overline B
\longrightarrow0
\]
and let \(\mathcal Y\) be a deformation of \(X_P\) over \(B\).
By induction, its restriction
$
\overline{\mathcal Y}
:=
\mathcal Y\otimes_B\overline B
$
is induced from \(\mathcal X_N\) by a homomorphism
$
\varphi_N\colon R_N\longrightarrow\overline B
$
for some \(N\).

Since \(R_N\) is a formal power-series ring in finitely many
variables, \(\varphi_N\) lifts to a homomorphism
\[
\widetilde\varphi_N\colon R_N\longrightarrow B.
\]
Let \(\mathcal Y_0\) be the corresponding pullback of
\(\mathcal X_N\). Then \(\mathcal Y\) and \(\mathcal Y_0\) are two
liftings of \(\overline{\mathcal Y}\) across the same small extension.
Their difference is represented by an element
$
\delta\in T^1_{X_P}\otimes_{\CC}J.
$

Since \(T^1_{X_P}\) is the algebraic direct sum of its homogeneous
components, \(\delta\) has finite homogeneous support. Hence there
exists \(N'\geq N\) such that
$
\delta\in V_{N'}\otimes_{\CC}J.
$
Regard \(\widetilde\varphi_N\) as a homomorphism
$
R_{N'}\longrightarrow B
$
by sending the additional parameters to zero. Since the
Kodaira--Spencer map of \(\mathcal X_{N'}\) is an isomorphism onto
\(V_{N'}\), there exists
\[
\lambda\in
\operatorname{Hom}_{\CC}
\left(
\mathfrak m_{R_{N'}}/
\mathfrak m_{R_{N'}}^2,
J
\right)
\]
whose image under the Kodaira--Spencer map is \(\delta\).

Because the extension is small, one has
$
\mathfrak m_BJ=0.
$
We may therefore modify the parameter homomorphism by setting
\[
\widetilde\varphi_{N'}'(q)
=
\widetilde\varphi_{N'}(q)+\lambda(\overline q)
\]
for each parameter \(q\), where \(\overline q\) denotes its class in
\(\mathfrak m_{R_{N'}}/\mathfrak m_{R_{N'}}^2\). This again defines a
local homomorphism
\[
\widetilde\varphi_{N'}'
\colon
R_{N'}\longrightarrow B
\]
lifting the original homomorphism to \(\overline B\). The corresponding
change of the pullback deformation is exactly \(\delta\). Hence the
pullback defined by \(\widetilde\varphi_{N'}'\) is isomorphic to
\(\mathcal Y\), proving the claim.

Finally, let
$
B'\longrightarrow B
$
be a small extension of local Artinian \(\CC\)-algebras and let
\(\mathcal Y\) be a deformation of \(X_P\) over \(B\). By the claim,
\(\mathcal Y\) is induced from \(\mathcal X_N\), for some \(N\), by a
homomorphism
$$
R_N\longrightarrow B.
$$
Since \(R_N\) is a formal power-series ring in finitely many variables,
this homomorphism lifts to a homomorphism
$
R_N\longrightarrow B'.
$
Pulling back \(\mathcal X_N\) gives a lifting of \(\mathcal Y\) to
\(B'\). Therefore
\[
\Def_{X_P}(B')
\longrightarrow
\Def_{X_P}(B)
\]
is surjective for every small extension \(B'\to B\). Thus
\(\Def_{X_P}\) is formally smooth, and \(X_P\) is unobstructed.
\end{proof}

\end{document}